\documentclass[11pt]{article}

\usepackage{amssymb}
\usepackage{amsfonts}
\usepackage{theorem}
\usepackage[leqno]{amsmath}
\usepackage{enumerate}
\usepackage{indentfirst, color}

\newenvironment{proof}[1][Proof]{\textbf{#1.} }{\hfill $\square$}

\newtheorem{lem}{Lemma}
\newtheorem{thm}{Theorem}
\newtheorem{defin}{Definition}
\newtheorem{rem}{Remark}
\newtheorem{prop}{Proposition}

{\begin{flushleft}\textbf{#1} \textbf{#2}} %
{\end{flushleft}}

\newcommand{\R}{\mathbb{R}}
\newcommand{\N}{\mathbb{N}}
\newcommand{\prb}{\mathbf{P}}
\newcommand{\E}{\mathbb{E}}
\newcommand{\tri}{\mathcal{F}}
\newcommand{\cG}{\mathcal{G}}
\newcommand{\cH}{\mathcal{H}}
\newcommand{\eps}{\varepsilon}
\newcommand{\phii}{\varphi}
\newcommand{\al}{\alpha}

\newcommand{\ind}{\mathbf{1}}

\newcommand{\tr}{\text{Trace}}
\newcommand{\diver}{\mbox{div}}
\newcommand{\tet}{\theta}

\newcommand{\wtil}{\widetilde}

\newcommand{\ope}{\mathcal{L}}
\newcommand{\supp}{\mbox{supp}}

\newcommand{\olaB}{\overleftarrow{dB_r}}
\newcommand{\bS}{\mathbb{S}}
\newcommand{\bH}{\mathbb{H}}
\newcommand{\bB}{\mathbb{B}}
\newcommand{\Prb}{\mathbf{P}}
\newcommand{\cE}{\mathcal{E}}
\newcommand{\cS}{\mathcal{S}}
\newcommand{\cR}{\mathcal{R}}

\title{Stochastic partial differential equations with singular terminal condition}
\author{A. Matoussi  \thanks{Universit\'e du
Maine, Institut du Risque et de l'Assurance du Mans, Laboratoire Manceau de Math\'ematiques, Avenue Olivier Messiaen, 72085 Le Mans, Cedex 9, France.} \thanks{Research partly supported by the Chair {\it Financial Risks} of the {\it Risk Foundation} sponsored by Soci\'et\'e G\'en\'erale, the Chair {\it Derivatives of the Future} sponsored by the F\'ed\'eration Bancaire Fran\c{c}aise, and the Chair {\it Finance and Sustainable Development }sponsored by EDF and Calyon}
\hspace{0.5cm} L. Piozin \footnotemark[1]  \hspace{0.5cm}  A. Popier \footnotemark[1] \thanks{Research partly supported by ANR STOSYMAP}
}
\date{\today}

\begin{document}

\maketitle

\vspace{1cm}

{\small \textbf{Abstract.} In this paper, we first prove existence and uniqueness of the solution of a backward doubly stochastic differential equation (BDSDE) and of the related stochastic partial differential equation (SPDE) under monotonicity assumption on the generator. Then we study the case where the terminal data is singular, in the sense that it can be equal to $+\infty$ on a set of positive measure. In this setting we show that there exists a minimal solution, both for the BDSDE and for the SPDE. Note that solution of the SPDE means weak solution in the Sobolev sense.

}

\vspace{0.3cm}

{\small \textit{Keywords:} backward doubly stochastic differential equations, stochastic partial differential equations, monotone condition, singular terminal data.

}

\section*{Introduction} 

Backward Doubly Stochastic Differential Equations (BDSDEs for short) have been introduced by Pardoux and Peng \cite{pard:peng:94} to provide a non-linear Feynman-Kac formula for classical solutions of SPDE. The main idea is to introduce in the standard BSDE a second nonlinear term driven by an external noise representing the random perturbation of the nonlinear SPDE. Roughly speaking, the BSDE becomes:
\begin{equation} \label{eq:BDSDE}
Y_{t} = \xi + \int_{t}^{T} f \left( r,Y_{r},Z_{r} \right) dr + \int_t^T g \left( r,Y_{r},Z_{r} \right) \overleftarrow{dB_r} - \int_{t}^{T} Z_{r}dW_{r}, \ 0 \leq t \leq T,
\end{equation}
where $B$ and $W$ are two independent Brownian motions, and $\overleftarrow{dB_r}$ is the backward It\^o integral. And the related class of SPDE is as follows: for $(t,x) \in [0,T] \times \R^{d}$ 
\begin{eqnarray} \label{eq:SPDE_gene}
u(t,x) &= & h(x) + \int_t^T \left[ \ope u(s,x) + f(s,x,u(s,x),(\sigma^* \nabla u)(s,x)) \right] ds \\ \nonumber
& + & \int_t^T g(s,x,u(s,x),(\nabla u \sigma)(s,x)) \overleftarrow{dB_r} .
\end{eqnarray} 
$\ope$ is a second-order differential operator:
$$\ope \phi = b \nabla \phi + \tr (\sigma \sigma^* D^2).$$
SPDE are PDE in which randomness is integrated to account for uncertainty. These equations appear naturally in various applications as for instance, Zakai equation in filtering (\cite{lipt:shir:01a,lipt:shir:01b}), in pathwise stochastic control theory or stochastic control with partial observations \cite{lion:soug:98}.  Pardoux and Peng \cite{pard:peng:94} have proven existence and uniqueness for solutions of BDSDE \eqref{eq:BDSDE} if $f$ and $g$ are supposed to be Lipschitz continuous functions and with square integrability condition on the terminal condition $\xi$ and on the coefficients $f(t,0,0)$ and $g(t,0,0)$. Moreover under smoothness assumptions of the coefficients, Pardoux and Peng prove existence and uniqueness of a classical solution for SPDE \eqref{eq:SPDE_gene} and the connection with solutions of BDSDE \eqref{eq:BDSDE}.

Several generalizations to investigate the weak solution of SPDE \eqref{eq:SPDE_gene} have been developed following different approaches:
\begin{itemize}
\item the technique of stochastic flow (Bally and Matoussi \cite{ball:mato:01}, Matoussi et al. \cite{mato:sabb:zhan:14, mato:sheu:02}, Kunita \cite{kuni:94a}, El Karoui and Mrad \cite{elkr:mrad:13});
\item the approach based on Dirichlet forms and their associated Markov processes (Denis and Stoica \cite{deni:stoi:04}, Bally et al. \cite{ball:pard:stoi:05}, Stoica \cite{stoi:03});
\item stochastic viscosity solution for SPDEs (Buckdahn and Ma \cite{buck:ma:01a,buck:ma:01b}, Lions and Souganidis \cite{lion:soug:00,lion:soug:01}).
\end{itemize}
Above approaches have allowed the study of numerical schemes for the Sobolev solution of semilinear SPDEs via Monte-Carlo methods (time discretization and regression schemes: \cite{bach:lasm:mato:13, bach:gobe:mato:15, mato:sabb:14}. For some general references on SPDE, see among others \cite{dapr:zabc:14,kryl:rozo:82,prev:rock:07,wals:86}. 


Popier in \cite{popi:06} studied the behavior of solutions of BSDE when the terminal condition is allowed to take infinite values on non-negligible set i.e.
\begin{equation*} 
\prb (\xi = +\infty \ \mbox{or} \ \xi = - \infty) > 0.
\end{equation*}
The generator $f$ is given by: $f(y) = -y |y|^{q}$. A minimal solution is constructed by an non decreasing approximation scheme. The main difficulty is the proof of continuity of the minimal solution $Y$ at time $T$. In general $Y$ is a ``supersolution'' and the converse property was proved under stronger sufficient conditions. 
Then in \cite{popi:06} the author established a link with viscosity solutions of the following PDE:
$$\frac{\partial}{\partial t} u+\ope u -u|u|^q=0.$$
$\ope$ is a second order differential operator defined by \eqref{eq:operator}. This PDE has been widely studied by PDE arguments (see among others Baras and Pierre \cite{bara:pier:84} and Marcus and Veron in \cite{marc:vero:99}). It is shown in \cite{marc:vero:99} that every solution of this PDE can be characterized by a final trace which is a couple $(\mathcal{S},\mu)$ where $\mathcal{S}$ is a closed subset of $\R^d$ and $\mu$ a non-negative Radon measure on $\R^d\backslash \mathcal{S}$. The final trace is attained in the following sense:
$$\lim_{t\rightarrow T}\int_{\R^d} u(t,x)\phi(x)dx=\left\{ \begin{array}{ll}
+\infty & \mbox{if } \supp(\phi) \cap \cS \neq \emptyset,\\
\displaystyle \int_{\R^d}\phi(x)\mu(dx) & \mbox{if } \supp(\phi) \subset \R^d\backslash\mathcal{S}.
\end{array} \right.$$
Dynkin and Kuznetsov \cite{dynk:kuzn:97} and Le Gall \cite{lega:96} have proved same kind of results but in a probabilistic framework by using the superprocesses theory.

Our main aim is to extend the results of \cite{ball:mato:01} and of \cite{popi:06} for the following SPDE with singular terminal condition $h$: for any $0\leq t \leq T$
\begin{eqnarray} \nonumber  
u(t,x) & = & h(x) + \int_t^T \left( \ope u(s,x) - u(s,x)|u(s,x)|^q \right) ds \\ \label{eq:sing_SPDE}
& & + \int_t^T g(s,u(s,x),\sigma(s,x)\nabla u(s,x)) \overleftarrow{dB_s}
\end{eqnarray}
where we will assume that $\cS = \{ h = +\infty\}$ is a closed non empty set. Roughly speaking we want to show that there is a (minimal) solution $u$ in the sense that 
\begin{itemize}
\item $u$ belongs to some Sobolev space and is a weak solution of the SPDE on any interval $[0,T-\delta]$, $\delta > 0$,
\item $u$ satisfies the terminal condition: $u(t,x)$ goes to $h(x)$ also in a weak sense as $t$ goes to $T$.
\end{itemize}

There are here two main difficulties. First we have to prove existence and uniqueness of the solution of a BDSDE with monotone generator $f$. To our best knowlegde the closest result on this topic is in Aman \cite{aman:12}. Nevertheless we think that there is a lack in this paper (precisely Proposition 4.2). Indeed for monotone BSDE ($g=0$) the existence of a solution relies on the solvability of the BSDE:
\begin{equation*} 
Y_{t} = \xi + \int_{t}^{T} f \left( r,Y_{r}\right) dr  - \int_{t}^{T} Z_{r}dW_{r}, \ 0 \leq t \leq T.
\end{equation*}
See among other the proof of Theorem 2.2 and Proposition 2.4 in \cite{pard:99}. To obtain a solution for this BSDE, the main trick is to truncate the coefficients with suitable truncation functions in order to have a bounded solution $Y$ (see Proposition 2.2 in \cite{bria:carm:00}). This can not be done for a general BDSDE. Indeed take for example ($\xi=f=0$ and $g=1$): 
\begin{equation*} 
Y_{t} =  \int_t^T \overleftarrow{dB_r} - \int_{t}^{T} Z_{r}dW_{r} = B_T - B_t , \ 0 \leq t \leq T,
\end{equation*}
with $Z=0$. Thus in order to prove existence of a solution for \eqref{eq:BDSDE}, one can not directly follow the scheme of \cite{pard:99}. And thus Proposition 4.2 in \cite{aman:12} is not proved. The first part of this paper is devoted to the existence of a solution for a monotone BDSDE (see Section \ref{sect:mono_BDSDE}) in the space $\cE^2$ (see Definition \ref{def:proc_space}). To realize this project we will restrict the class of functions $f$: they should satisfy a polynomial growth condition (as in \cite{bria:carm:00}). Until now we do not know how to extend this to general growth condition as in \cite{bria:dely:hu:03} or \cite{pard:99}. Moreover under monotonicity assumption on $f$, we also prove that the SPDE \eqref{eq:SPDE_gene} has a unique weak solution (as in \cite{ball:mato:01}).

The second goal of this work is to extend the results of \cite{popi:06} to the doubly stochastic framework. We will consider the generator $f(y)=-y|y|^{q}$ with $q \in \R_+^*$ and a real $\tri^W_{T}$-measurable and non negative random variable $\xi$ such that:
\begin{equation}\label{cond_infinity}
\prb (\xi = +\infty) > 0.
\end{equation}
And we want to find a solution to the following BDSDE:
\begin{equation} \label{eq:sing_BDSDE} 
Y_{t}  =  \xi - \int_{t}^{T} Y_{s}|Y_s|^q ds + \int_t^T g(s,Y_s,Z_s) \overleftarrow{dB_s}  - \int_{t}^{T} Z_{s} dW_{s}.
\end{equation}
The scheme to construct a solution is almost the same as in \cite{popi:06}. Let us emphasize one of main technical difficulties. If $g=0$, we can use the conditional expectation w.r.t. $\tri_t$ to withdraw the martingale part. If $g \neq 0$, this trick is useless and we have to be very careful when we want almost sure property of the solution.

Finally this BDSDE is connected with the stochastic PDE \eqref{eq:sing_SPDE} with singular terminal condition $h$. From the first part of this paper, if $h$ is in $L^{2(q+1)}(\R^d,\rho^{-1}dx)$ or if $\xi = h(X^{t,x}_T)$ belongs to $L^{2(q+1)}(\Omega)$, then there exists a unique weak solution to \eqref{eq:sing_SPDE}. Our aim is to extend this when \eqref{cond_infinity} holds. 

The paper is decomposed as follows. In the first section, we give the mathematical setting and our main contributions. In the next section, we study the existence and uniqueness of a class of monotone BDSDE and SPDE. In Section \ref{sect:exist_min_sol} we construct a (super)solution for the BDSDE with singular terminal condition. In Section \ref{sect:cont} we prove continuity at time $T$ for this solution under sufficient conditions. Finally in the last part, we connect BDSDE and SPDE with a singularity at time $T$. 

\section{Setting and main results} \label{sect:main_results}

Let us now precise our notations. $W$ and $B$ are independent Brownian motions defined on a probability space $(\Omega,\tri,\Prb)$ with values in $\R^d$ and $\R^m$. Let $\mathcal{N}$ denote the class of $\Prb$-null sets of $\tri$. For each $t\in [0,T]$, we define
$$\tri_t = \tri^W_t \vee \tri_{t,T}^B$$
where for any process $\eta$, $\tri^\eta_{s,t} =\sigma \left\{ \eta_r - \eta_s; \ s \leq r \leq t \right\} \vee \mathcal{N}$, $\tri^\eta_t = \tri^\eta_{0,t}$. As in \cite{pard:peng:94} we define the following filtration $(\cG_t, \ t\in [0,T])$ by:
$$\cG_t =  \tri^W_t \vee \tri_{0,T}^B.$$
$\xi$ is a $\tri^W_{T}$-measurable and $\R^N$-valued random variable.

We define by $\bH^p(0,T;\R^N)$ the set of (classes of $d\Prb \times dt$ a.e. equal) $N$-dimensional jointly measurable randon processes $(X_t, \ t \geq 0)$ which satisty:
\begin{enumerate}
\item $\displaystyle \E \left(  \int_0^T |X_t|^2 dt \right)^{p/2} < +\infty$
\item $X_t$ is $\cG_t$-measurable for a.e. $t\in [0,T]$. 
\end{enumerate}
We denote similarly by $\bS^p(0,T;\R^N)$ the set of continuous $N$-dimensional random processes which satisfy:
\begin{enumerate}
\item $\displaystyle \E \left(  \sup_{t\in [0,T]} |X_t|^p  \right) < +\infty$
\item $X_t$ is $\cG_t$-measurable for any $t\in [0,T]$. 
\end{enumerate}
\begin{defin} \label{def:proc_space}
$\ $
\begin{itemize}
\item $\bB^p(0,T)$ is the product $\bS^p(0,T;\R^N) \times \bH^p(0,T;\R^{N\times d})$. 
\item $(Y,Z) \in \cE^p(0,T)$ if $(Y,Z)\in \bB^p(0,T)$ and $Y_t$ and $Z_t$ are $\tri_t$-measurable. 
\end{itemize}
\end{defin}
Finally $C^{p,q}([0,T]\times \R^d; \R^k)$ denotes the space of $\R^k$-valued functions defined on $[0,T]\times \R^d$ which are $p$-times continuously differentiable in $t \in [0,T]$ and $q$-times continuously differentiable in $x \in \R^d$. $C^{p,q}_b([0,T]\times \R^d; \R^k)$ is the subspace of $C^{p,q}([0,T]\times \R^d; \R^k)$ in which all functions have uniformly bounded partial derivatives; and $C^{p,q}_c([0,T]\times \R^d; \R^k)$ the subspace of $C^{p,q}([0,T]\times \R^d; \R^k)$ in which the functions have a compact support w.r.t. $x \in \R^d$.

Now we precise our assumptions on $f$ and $g$. The functions $f$ and $g$ are defined on $[0,T] \times \Omega \times \R^N \times \R^{N\times d}$ with values respectively in $\R^{N}$ and $\R^{N \times m}$. Moreover we consider the following \textbf{Assumptions (A)}.
\begin{itemize}
\item The function $y\mapsto f(t,y,z)$ is continuous and there exists a constant $\mu$ such that for any $(t,y,y',z)$ a.s.
\begin{equation} \label{ineq:monoton} \tag{A1}
\langle y-y', f(t,y,z)- f(t,y',z) \rangle \leq \mu |y-y'|^2 .
\end{equation}
\item There exists $K_f$ such that for any $(t,y,z,z')$ a.s.
\begin{equation} \label{ineq:lip_Z}  \tag{A2}
|f(t,y,z)-f(t,y,z')|^2 \leq K_f |z-z'|^2.
\end{equation}
\item There exists $C_f\geq 0$ and $p > 1$ such that 
\begin{equation} \label{eq:property_growth_f}  \tag{A3}
|f(t,y,z)-f(t,0,z)| \leq C_f(1+|y|^p).
\end{equation}
\item There exists a constant $K_g \geq 0$ and $0 < \eps < 1$ such that for any $(t,y,y',z,z')$ a.s.
\begin{equation} \label{eq:property_Lip_g}\tag{A4}
|g(t,y,z) - g(t,y',z') |^2 \leq K_g |y-y'|^2 + \eps |z-z'|^2.
\end{equation}
\item Finally $(f(t,0,0), \ t \geq 0)$ and $(g(t,0,0), \ t \geq 0)$ are $\tri_t$ measurable with for some $p> 1$
\begin{equation} \label{eq:L2_integ} \tag{A5}
\E \int_0^T\left(  |f(t,0,0)|^2 +  |g(t,0,0)|^2 \right) dt < +\infty.
\end{equation} 
\end{itemize}
Remember that from \cite{pard:peng:94} if $f$ also satisfies: there exists $\tilde K_f$ such that for any $(t,y,y',z)$ a.s.
\begin{equation} \label{eq:property_Lip_f}
|f(t,y,z) - f(t,y',z) | \leq \tilde K_f |y-y'|
\end{equation}
and if $\xi \in L^2(\Omega)$, then there exists a unique solution $(Y,Z)\in \cE^2(0,T)$ to the BDSDE \eqref{eq:BDSDE}. Note that \eqref{eq:property_Lip_f} implies that 
$$|f(t,y,z)-f(t,0,z)| \leq \tilde K_f |y|,$$
thus the growth assumption \eqref{eq:property_growth_f} on $f$ is satisfied with $p=1$. 

In Section \ref{sect:mono_BDSDE} we will prove the following result. 
\begin{thm} \label{thm:existence_sol}
Under assumptions (A) and if the terminal condition $\xi$ satisfies
\begin{equation} \label{eq:term_cond_L2}
\E (|\xi|^2) <+\infty,
\end{equation} 
the BDSDE \eqref{eq:BDSDE} has a unique solution $(Y,Z) \in \cE^2(0,T)$. Moreover if for some $p \geq 1$
\begin{equation} \label{eq:L_2p_condition}
\E \left[ |\xi|^{2p} + \left( \int_0^T\left(  |f(t,0,0)|^2 +  |g(t,0,0)|^2 \right) dt \right)^p \right] <+\infty,
\end{equation} 
then $(Y,Z) \in \cE^{2p}(0,T)$.
\end{thm}
Using the paper of Aman \cite{aman:12}, this result can be extended to the $L^p$ case: for $p\in (1,2)$, if
\begin{equation*} 
\E \left[ |\xi|^p + \int_0^T\left(  |f(t,0,0)|^p +  |g(t,0,0)|^p \right) dt \right]  <+\infty,
\end{equation*} 
there exists a unique solution in $ \cE^p(0,T)$. We also give (for completeness) a comparison result on the solution of the BDSDE \eqref{eq:BDSDE} and finally we extend Theorem 3.1 in \cite{ball:mato:01} in the monotone case. Let us precise here the setting. 
For all $(t,x) \in [0,T] \times \R^{d}$, we denote by $X^{t,x}$ the solution of the following SDE:
\begin{equation} \label{eq:eds2}
X^{t,x}_{s} = x + \int_{t}^{s} b(X^{t,x}_{r})dr + \int_{t}^{s} \sigma(X^{t,x}_{r})dW_{r}, \ \text{for} \ s \in [t,T],
\end{equation}
and $X^{t,x}_{s} = x \ \text{for} \ s \in [0,t]$. We assume that $b \in C^{2}_b(\R^d;\R^d)$ and $\sigma \in C^3_b(\R^d;\R^{d\times d})$. We consider the following doubly stochastic BSDE for $t \leq s \leq T$:
\begin{equation} \label{eq:fbsde_gene} 
Y^{t,x}_{s}  = h(X^{t,x}_{T}) + \int_{s}^{T} f(r,X^{t,x}_r,Y^{t,x}_{r},Z^{t,x}_{r}) dr + \int_s^T g(r,X^{t,x}_r,Y^{t,x}_r,Z^{t,x}_r) \overleftarrow{dB_r}  - \int_{s}^{T} Z^{t,x}_{r} dW_{r},
\end{equation}
where $h$ is a function defined on $\R^{d}$ with values in $\R$. In order to define the space of solutions, we choose a continuous positive weight function $\rho :\R^d \to \R$. We require only that the derivatives of $\rho$ are in $C^{1}_b(\R^d; \R)$ on the set $\{|x|> R\}$ for some $R$. For example $\rho$ can be $(1+|x|)^\kappa$, $\kappa \in \R$. We assume that the functions $f : [0,T]\times \R^d \times \R^N \times \R^{N\times d} \to \R^N$ and $g : [0,T]\times \R^d \times \R^N \times \R^{N\times d} \to \R^{N\times m}$ are measurable in $(t,x,y,z)$ and w.r.t. $(y,z)$, $f$ and $g$ satisfy Assumptions \eqref{ineq:monoton} to \eqref{eq:property_Lip_g}. The only difference with \cite{ball:mato:01} is that we do not assume that $f$ is Lipschitz continuous w.r.t. $y$. We also assume that 
\begin{equation} \label{eq:cond_L_2p_SPDE}
 \int_{\R^d} \left[|h(x)|^{2p} +  \int_0^T \left( |f(t,x,0,0)|^{2p} + |g(t,x,0,0)|^{2p} \right) dt \right]\rho^{-1}(x) dx  < +\infty.
 \end{equation}
We define the space $\cH(0,T)$ as in \cite{ball:mato:01}. 
\begin{defin} \label{def:sobolev_space}
$\cH(0,T)$ is the set of the random fields $\{u(t,x); \ 0\leq t \leq T, \ x \in \R^d\}$ such that $u(t,x)$ is $\tri_{t,T}^B$-measurable for each $(t,x)$, $u$ and $\sigma^* \nabla u$ belong to $L^2((0,T)\times \Omega\times \R^d; \ ds \otimes d \prb \otimes \rho^{-1}(x) dx)$. On $\cH(0,T)$ we consider the following norm
$$\| u \|_2^2 =\E   \int_{\R^d} \int_0^T \left( |u(s,x)|^2 + |(\sigma^* \nabla u)(s,x)|^2 \right) \rho^{-1}(x)  ds dx.$$
\end{defin}
We summarize our result in the next proposition. 
\begin{prop} \label{prop:existence_sol_monotone_SPDE}
Under assumptions (A) and if Condition \eqref{eq:cond_L_2p_SPDE} holds, then the random field defined by $u(t,x) = Y^{t,x}_t$ is in $\cH(0,T)$ with
\begin{equation} \label{eq:2p_estim}
\E  \int_{\R^d}  \int_0^T  |u(s,x)|^{2p} \rho^{-1}(x) ds  dx.
\end{equation}
Moreover $u$ is the unique weak solution (see Definition \ref{def:weak_sol_SPDE}) of the SPDE \eqref{eq:SPDE_gene}.
\end{prop}
Note that Condition \eqref{eq:cond_L_2p_SPDE} is important to ensure that \eqref{eq:2p_estim} holds and therefore the quantity $F_s =f(s,x,u(s,x),(\sigma^* \nabla u)(s,x))$ is in $\cH_{0,\rho}'$, which is crucial to prove the existence of a weak solution.

The next sections are devoted to the singular case. The generator $f$ will be supposed to be deterministic and given by: $f(y) = -y|y|^q$ for some $q> 0$. The aim is to prove existence of a solution for BDSDE \eqref{eq:sing_BDSDE} when the non negative random variable $\xi$ satisfies \eqref{cond_infinity}. A possible extension of the notion of solution for a BDSDE with singular terminal condition could be the following (see Definition 1 in \cite{popi:06}). 
\begin{defin}[Solution of the BDSDE \eqref{eq:sing_BDSDE}] \label{definsolution}
Let $q > 0$ and $\xi$ a $\tri^W_{T}$-measurable non negative random variable satisfying condition \eqref{cond_infinity}. We say that the process $(Y,Z)$ is a solution of the BDSDE \eqref{eq:sing_BDSDE} if $(Y,Z)$ is such that $(Y_t,Z_t)$ is $\tri_t$-measurable and:
\begin{enumerate}
\item[(D1)] for all $0 \leq s \leq t < T$: $\displaystyle Y_{s} = Y_{t} - \int_{s}^{t} Y_{r} |Y_{r}|^{q} dr + \int_s^t g(r,Y_r,Z_r) \overleftarrow{dB_r} - \int_{s}^{t} Z_{r} dW_{r}$;
\item[(D2)] for all $t \in [0,T[$, $\displaystyle \E \left( \sup_{0\leq s \leq t} |Y_{s}|^{2} + \int_{0}^{t} 
\| Z_{r} \|^{2} dr \right) < + \infty$;
\item[(D3)] $\prb$-a.s. $\displaystyle \lim_{t \to T} Y_{t} = \xi$.
\end{enumerate}
A solution is said non negative if a.s. for any $t \in [0,T]$, $Y_t \geq 0$. 
\end{defin}

To obtain an a priori estimate of the solution we will assume that $g(t,y,0)=0$ for any $(t,y)$ a.s. This condition will ensure that our solutions will be non negative and bounded on any time interval $[0,T-\delta]$ with $\delta > 0$. Without this hypothesis, integrability of the solution would be more challenging. In Section \ref{sect:exist_min_sol}, we will prove the following result.
\begin{thm} \label{thm:exist_min_sol} 
There exists a process $(Y,Z)$ satisfying Conditions (D1) and (D2) of Definition \ref{definsolution} and such that $Y$ has a limit at time $T$ with 
$$\lim_{t \to T} Y_t \geq \xi.$$
Moreover this solution is minimal: if $(\widetilde Y, \widetilde Z)$ is a non negative solution of \eqref{eq:sing_BDSDE}, then a.s. for any $t$, $\widetilde Y_t \geq Y_t$.
\end{thm}
It means in particular that $Y_t$ has a left limit at time $T$. 

In general we are not able to prove that (D3) holds. As in \cite{popi:06}, we give sufficient conditions for continuity and we prove it in the Markovian framework. Hence the first hypothesis on $\xi$ is the following:
\begin{equation} \tag{H1} \label{eq:hyp1}
\xi = h ( X_{T}),
\end{equation}
where $h$ is a function defined on $\R^{d}$ with values in $\overline{\R^{+}}$ such that the set of singularity $\cS = \left\{ h=+\infty \right\}$ is closed; and where $X_{T}$ is the value at $t=T$ of a diffusion process or more precisely the solution of a stochastic differential equation (in short SDE):
\begin{equation} \label{eq:eds}
X_{t} = x + \int_{0}^{t} b(r,X_{r})dr + \int_{0}^{t} \sigma(r,X_{r})dW_{r}, \ \text{for} \ t \in [0,T].  
\end{equation}
We will always assume that $b$ and $\sigma$ are defined on $[0,T] \times \R^{d}$, with values respectively in $\R^{d}$ and $\R^{d \times k}$, are measurable w.r.t. the Borelian $\sigma$-algebras, and that there exists a constant $K > 0$ s.t. for all $t \in [0,T]$ and for all $(x,y) \in \R^{d} \times \R^{d}$:
\begin{enumerate}
\item Lipschitz condition:
\begin{equation} \tag{L} \label{eq:lipcond}
|b(t,x)-b(t,y)| + |\sigma(t,x)-\sigma(t,y)| \leq K |x-y|;
\end{equation}
\item Growth condition:
\begin{equation} \tag{G} \label{eq:growthcond}
|b(t,x)| \leq K(1+|x|) \ \mbox{and} \  |\sigma(t,x)| \leq K(1+|x|).
\end{equation}
\end{enumerate}
It is well known that under the previous assumptions, Equation \eqref{eq:eds} has a unique strong solution $X$. 

We denote $\cR = \R^{d} \setminus \cS$. The second hypothesis on $\xi$ is: for all compact set $\mathcal{K} \subset \cR = \R^{d} \setminus \left\{ h = + \infty \right\}$ 
\begin{equation} \tag{H2} \label{eq:hyp2}
h(X_{T}) \mathbf{1}_{\mathcal{K}}(X_{T}) \in \ L^{1} \left( \Omega, \tri_{T}, \prb ;\R \right).
\end{equation}

Unfortunately the above assumptions are not sufficient to prove continuity if $q \leq 2$. Thus we add the following conditions in order to use Malliavin calculus and to prove Equality \eqref{eq:malliavin_int_part}.
\begin{enumerate}
\item The functions $\sigma$ and $b$ are bounded: there exists a constant $K$ s.t.  
\begin{equation} \tag{B} \label{boundcond}
\forall (t,x) \in [0,T] \times \R^{d}, \quad |b(t,x)| + |\sigma (t,x) | \leq K.
\end{equation}
\item The second derivatives of $\sigma \sigma^{*}$ belongs to $L^{\infty}$:
\begin{equation} \tag{D} \label{secondderiv}
\frac{\partial^{2} \sigma \sigma^{*}}{\partial x_{i} \partial x_{j}} \in L^{\infty}([0,T] \times \R^{d}).
\end{equation}
\item $\sigma \sigma^{*}$ is uniformly elliptic, i.e. there exists $\lambda > 0$ s.t. for all $(t,x) \in [0,T] \times \R^{d}$:
\begin{equation} \tag{E} \label{ellipcond}
\forall y \in \R^{d}, \ \sigma \sigma^{*}(t,x)y.y \geq \lambda |y|^{2}.
\end{equation}
\item $h$ is continuous from $\R^{d}$ to $\overline{\R_{+}}$ and: 
\begin{equation} \tag{H3} \label{eq:hyp3}
\forall M \geq 0, \ h \ \mbox{is a Lipschitz function on the set} \ \mathcal{O}_{M} = \left\{ |h| \leq M \right\}.
\end{equation}
\end{enumerate}
\begin{thm}\label{thm:continuity_T}
Under assumptions \eqref{eq:hyp1}, \eqref{eq:hyp2} and \eqref{eq:lipcond} and if 
\begin{itemize}
\item either $q>2$ and \eqref{eq:growthcond};
\item or \eqref{boundcond}, \eqref{secondderiv}, \eqref{ellipcond} and \eqref{eq:hyp3};
\end{itemize}
the minimal non negative solution $(Y,Z)$ of \eqref{eq:sing_BDSDE} satisfies (D3): a.s.
$$\lim_{t\to T} Y_t = \xi.$$
\end{thm}

Finally in section \ref{sect:SPDE}, we show that this minimal solution $(Y,Z)$ of \eqref{eq:sing_BDSDE} is connected to the minimal weak solution $u$ of the SPDE \eqref{eq:sing_SPDE}. More precisely $X^{t,x}$ is the solution of the SDE \eqref{eq:eds2}
with initial condition $x$ at time $t$ and $(Y^{t,x},Z^{t,x})$ is the minimal solution of the BDSDE \eqref{eq:sing_BDSDE} with singular terminal condition $\xi = h(X^{t,x}_T)$. 
\begin{thm}\label{thm:sol_SPDE}
The randon field $u$ defined by $u(t,x) = Y^{t,x}_t$ belongs to $\cH(0,T-\delta)$ for any $\delta > 0$ and is a weak solution of the SPDE \eqref{eq:sing_SPDE} on $[0,T-\delta]\times \R^d$. At time $T$, $u$ satisfies a.s. $\liminf_{t\to T} u(t,x) \geq h(x)$. 

Moreover under the same assumptions of Theorem \ref{thm:continuity_T}, for any function $\phi \in C^{\infty}_c(\R^d)$ with support included in $\cR$, then
$$\lim_{t\to T} \E\left( \int_{\R^d} u(t,x)\phi(x) dx \right) = \int_{\R^d} h(x) \phi(x) dx.$$

Finally $u$ is the minimal non negative solution of \eqref{eq:sing_SPDE}.
\end{thm}
The almost sure continuity of $u$ at time $T$ is still an open question. In \cite{popi:06}, this property is proved using viscosity solution arguments (relaxation of the boundary condition). Here we cannot do the same trick. This point will be investigated in further publications.

\textit{In the continuation, unimportant constants will be denoted by $C$.}

\section{Monotone BDSDE and SPDE} \label{sect:mono_BDSDE}

As mentioned in the introduction and in the previous section, our first contribution is the extension of the result of Pardoux and Peng \cite{pard:peng:94} with monotone condition \eqref{ineq:monoton}. We begin with the particular case where $f$ does not depend on $z$ and $g$ is a given random field.

\subsection{Case with $f(t,y,z)=f(t,y)$ and $g(t,y,z)=g_t$}

In this special case assume that there exists a solution to the BDSDE: 
\begin{equation} \label{eq:BDSDE_1}
Y_{t} = \xi + \int_{t}^{T} f \left( r,Y_{r} \right) dr + \int_t^T g_r \overleftarrow{dB_r} - \int_{t}^{T} Z_{r}dW_{r}, \ 0 \leq t \leq T.
\end{equation}
Then we have
\begin{equation*} 
Y_{t} + \int_0^t g_r \overleftarrow{dB_r}  = \xi + \int_0^T  g_r \overleftarrow{dB_r} + \int_{t}^{T} f \left( r,Y_{r} \right) dr - \int_{t}^{T} Z_{r}dW_{r}, \ 0 \leq t \leq T.
\end{equation*}
Let us define:
$$U_t = Y_{t} + \int_0^t g_r \overleftarrow{dB_r},\qquad \zeta = \xi + \int_0^T g_r \overleftarrow{dB_r},$$
and
$$\phi(t,y) = f\left( t,y- \int_0^t g_r \olaB \right).$$
Then $(U,Z)$ satisfies:
\begin{equation} \label{eq:auxil_BSDE}
U_{t}  =\zeta + \int_{t}^{T} \phi \left( r,U_{r} \right) dr - \int_{t}^{T} Z_{r}dW_{r}, \ 0 \leq t \leq T.
\end{equation}
The terminal condition $\zeta$ is $\cG_T$-measurable and the generator $\phi$ satisfies the following assumptions.
\begin{enumerate}
\item $\phi$ is continuous w.r.t. $y$ and \eqref{ineq:monoton} is true with the same constant $\mu$.
\item From \eqref{eq:property_growth_f}, there exists $p> 1$ such that  
\begin{equation} 
|\phi (t,y)| \leq h(t) + C_{\phi}(1+|y|^p). 
\end{equation}
where $C_{\phi}= C_f 2^{p-1}$ and
$$h(t) = |f(t,0)| +2^{p-1} \left|\int_0^t g_r \olaB\right|^p.$$
\end{enumerate}
On the solution $(U,Z)$ we impose the following measurability constraints:
\begin{enumerate}
\item[(M1).] The process $(U,Z)$ is adapted to the filtration $(\cG_t, \ t \geq 0)$.
\item[(M2).] The random variable $U_t - \int_0^t g_r \olaB$ is $\tri_t$-measurable for any $0\leq t \leq T$.
\end{enumerate}

Let us assume the boundedness hypothesis on $\xi$, $g$ and $f(t,0)$: there exists a constant $\gamma > 0$ such that a.s. for any $t \geq 0$, 
\begin{equation} \label{eq:bound_cond}
|\xi| + |f(t,0)| + |g_t| \leq \gamma.
\end{equation}
Hence for any $q > 1$
$$\E \left[ |\zeta|^{q} +\left( \int_0^T |h(t)|^{q} dt \right) \right] < +\infty.$$
From \cite{bria:carm:00} or \cite{pard:99} there exists a unique solution $(U,Z)\in \bB^{2}(0,T)$ to the BSDE \eqref{eq:auxil_BSDE} such that (M1) holds and 
$$\E \left[ \sup_{t \in [0,T]} |U_t|^{2} +\left(  \int_0^T |Z_r|^2 dr \right) \right] < +\infty.$$
Theorem 3.6 in \cite{bria:carm:00} also gives that 
\begin{equation}\label{eq:L_2p_estim}
\E \left[ \sup_{t \in [0,T]} |U_t|^{2p} +\left(  \int_0^T |Z_r|^2 dr \right)^p \right] < +\infty.
\end{equation}
But we cannot derive directly from this result that (M2) is satisfied, that is $U_t - \int_0^t g_r \olaB$ is $\tri_t$-measurable for any $0\leq t \leq T$. Therefore we follow the proof of Proposition 3.5 in \cite{bria:carm:00} to prove the existence and uniqueness of the solution $(U,Z)$ with the desired measurability conditions. 
\begin{prop}
Under Assumptions \eqref{ineq:monoton}, \eqref{ineq:lip_Z}, \eqref{eq:property_growth_f} and \eqref{eq:bound_cond}, there exists a unique solution $(U,Z) \in \bB^{2p}(0,T)$ to the BSDE \eqref{eq:auxil_BSDE}, such that (M1) and (M2) hold.
\end{prop}
\begin{proof}
As written before, we sketch the proof of Proposition 3.5  in \cite{bria:carm:00}. The details can be found in \cite{bria:carm:00} and we just emphasize the main differences. For any $n \geq 1$, we define the following function:
$$\wtil \phi_n(t,y) = \left\{ \begin{array}{cl}
\phi(t,y) & \mbox{if } h(t)\leq n ,\\
\frac{n}{|h(t)|} \phi(t,y) & \mbox{if } h(t) > n.
\end{array} \right.$$
This function is continuous w.r.t. $y$ and \eqref{ineq:monoton} still holds. Moreover
$$|\wtil \phi_n(t,y) | \leq (|h(t)|\wedge n) + C_{\phi}(1+|y|^p).$$
Then as in \cite{bria:carm:00}, we define
$$\phi_n(t,.) = \rho_n * (\Theta_{q(n)+1} (\wtil \phi_n(t,.))),$$
where
\begin{itemize}
\item $q(n) = \lceil e^{1/2} (n+2C_{\phi})\sqrt{1+T^2} \rceil +1$, where $\lceil r \rceil$ stands for the integer part of $r$;
\item $\Theta_n$ is a smooth function with values in $[0,1]$ such that $\Theta_n(u) = 1$ if $|u|\leq n$, $\Theta_n(u) = 0$ if $|u|\geq n+1$;
\item $\rho_n(u) = n^k \rho(n u)$ with $\rho$ a $C^\infty$ non negative function with support equal to the unit ball and such that $\int \rho (u) du = 1$.
\end{itemize}
Since $\zeta$ is in $L^q(\Omega)$ (for any $q > 2p$) there exists a unique solution $(U^n,V^n) \in \bB^q(0,T)$ to the BSDE (see Theorem 4.2 in \cite{bria:dely:hu:03} or Theorem 5.1 in \cite{elka:peng:quen:97}):
\begin{equation} \label{eq:auxil_BSDE_n}
U^n_{t}  =\zeta  + \int_{t}^{T} \phi_n \left( r,U^n_{r} \right) dr - \int_{t}^{T} V^n_{r}dW_{r}, \ 0 \leq t \leq T.
\end{equation}
Moreover for some constant $K_p$ independent of $n$
$$\E \left[ \sup_{t \in [0,T]} |U^n_t|^{2p} +\left(  \int_0^T |V^n_r|^2 dr \right)^p \right] \leq K_p \E \left[  |\zeta |^{2p} +\left(  \int_0^T (|h(r)|+2C_{\phi}) dr \right)^{2p} \right] .$$
We have a strong convergence of the sequence $(U^n,V^n)$ to $(U,Z)$:
$$\lim_{n\to +\infty} \E \left[ \sup_{t \in [0,T]} |U^n_t-U_t|^{2} +\left(  \int_0^T |V^n_r-Z_r|^2 dr \right) \right] = 0.$$
And $(U,Z)$ is the solution of BSDE \eqref{eq:auxil_BSDE} satisfying condition (M1) and $(U,Z) \in \bB^{2p}(0,T)$.

Now let us come to the measurability condition (M2). Recall that 
\begin{eqnarray*}
\phi(t,y) & = & f\left( t,y- \int_0^t g_r \olaB \right),\\
h(t) & = & |f(t,0)| +2^{p-1} \left| \int_0^t g_r \olaB\right|^p,
\end{eqnarray*}
and the process $f(t,.)$ is $\tri_t$-measurable. Hence for any $y$ and $t$
\begin{eqnarray*}
\phi_n(t,y+\int_0^t g_r \olaB) & = & \rho_n * (\Theta_{q(n)+1} (\wtil \phi_n(t,.)))(y+\int_0^t g_r \olaB) \\
& =&  \int \rho_n(z) \Theta_{q(n)+1} \left( \wtil \phi_n\left(t,y-z+\int_0^t g_r \olaB\right)\right) dz. 
\end{eqnarray*}
Now
$$\wtil \phi_n\left(t,x+\int_0^t g_r \olaB\right) = \frac{n}{h(t)\vee n}f(t,x) = \frac{n}{h(t) \ind_{h(t) \geq n} \vee n} f(t,x),$$
Thus $\phi_n(t,y+\int_0^t g_r \olaB) $ is measurable w.r.t. the $\sigma$-algebra $\tri_t \vee \sigma(h(t) \ind_{h(t) \geq n})$. Let us denote
$$\cH^n = \sigma(h(t) \ind_{h(t) \geq n}, \ 0\leq t\leq T).$$
$\phi_n(t,y+\int_0^t g_r \olaB) $ is measurable w.r.t. the $\sigma$-algebra $\tri^n_t=\tri_t \vee \cH^n $. 

If we define 
$$Y^n_t = U^n_t - \int_0^t g_r \olaB,$$
then 
$$Y^n_t = \xi + \int_t^T g_r  \overleftarrow{B_r} + \int_{t}^{T} \phi_n \left( r,Y^n_{r} + \int_0^r g_s \overleftarrow{B_s} \right) dr - \int_{t}^{T} V^n_{r}dW_{r}, \ 0 \leq t \leq T.$$
We claim that $Y^n_t$ is measurable w.r.t. $\tri_t \vee \cH^n$. Indeed let us recall that $(U^n,V^n)$, solution of \eqref{eq:auxil_BSDE_n}, is obtained via a fixed-point theorem. We define the map $\Psi : \bB^{2}(0,T) \to \bB^2(0,T)$ by: $(U,V) = \Psi(u,v)$ with
\begin{equation*} 
U_{t}  =\zeta  + \int_{t}^{T} \phi_n \left( r,u_{r} \right) dr - \int_{t}^{T} V_{r}dW_{r}, \ 0 \leq t \leq T.
\end{equation*}
By classical arguments (see the details in \cite{elka:peng:quen:97}, Theorem 2.1), $\Psi$ is a contraction on $\bB^{2}(0,T)$ (under suitable norms) and $(U^n,V^n)$ is the fixed point of $\Psi$. Set $(U^{n,m},V^{n,m})$ for any $m\in \N$ as follows:  for any $t$, $(U^{n,0}_t,V^{n,0}_t)=(\int_0^t g_s \overleftarrow{B_s},0)$ and for any $m \geq 1$, $(U^{n,m},V^{n,m})= \Psi(U^{n,m-1},V^{n,m-1})$. This sequence converges in $\bB^{2}(0,T)$ to $(U^n,V^n)$. 
Therefore $Y^n$ is the limit in $\bS^2(0,T)$ of $Y^{n,m}$ defined by: $Y^{n,0}_t = 0$ and 
\begin{equation*} 
Y^{n,m}_{t}  =\xi + \int_t^T g_r  \overleftarrow{B_r}  + \int_{t}^{T} \phi_n  \left( r,Y^{n,m-1}_{r} + \int_0^r g_s \overleftarrow{B_s} \right) dr - \int_{t}^{T} V^{n,m}_{r}dW_{r}, \ 0 \leq t \leq T.
\end{equation*}
Now $Y^{n,0}_t$ is trivially $\tri_t$ measurable and 
\begin{eqnarray*} 
Y^{n,m}_{t} & = & \E \left[ \xi + \int_t^T g_r  \overleftarrow{B_r} \bigg| \cG_t \right]  + \E \left[ \int_{t}^{T} \phi_n  \left( r,Y^{n,m-1}_{r} + \int_0^r g_s \overleftarrow{B_s} \right) dr  \bigg| \cG_t \right] \\
& =& \E \left[ \xi + \int_t^T g_r  \overleftarrow{B_r} \bigg| \cG_t \right]  + \E \left[ \Theta  \big| \cG_t \right].
\end{eqnarray*}
From \cite{pard:peng:94} we know that the first term on the right hand side is $\tri_t$ measurable. Assume that $Y^{n,m-1}_{t}$ is $\tri_t \vee \cH^n$ measurable. Since the same holds for $\phi_n  \left( t,y+ \int_0^t g_s \overleftarrow{B_s} \right)$, $\Theta$ depends only on $ \tri^W_T \vee \tri^B_{t,T} \vee \cH^n$. Thus there is no independence between $\tri_t^B$ and $\tri_t \vee \sigma(\Theta)$, but $Y^{n,m}_t$ depends on $\tri_t^B$ only through $\cH^n$. Hence $Y^{n,m}_t$ is $\tri_t \vee \cH^n$ measurable. Passing through the limit, we obtain the desired measurability condition on $Y^n$. 

Now for any $m\in \N$, the sequence $(Y^n_t, \ n\geq m)$ depends on the $\sigma$-algebra
$$\tri_t \vee \overline{\cH_m} = \tri_t \vee \left(  \bigvee_{n\geq m} \cH^n \right).$$
Passing through the limit, we obtain that the limit $Y_t$ depends only on $\displaystyle  \tri_t \vee \cH^\infty= \tri_t \vee \bigcap_{m\in \N} \overline{\cH_m}$. The next lemma shows that $\cH^\infty \subset \tri_0$. We deduce that $Y_t$ is $\tri_t$-measurable, which achieves the proof.
\end{proof}

\begin{lem}
The $\sigma$-algebra $\cH^\infty$ is trivial: for every $A \in \cH^\infty$, $A$ or $A^c=\Omega\setminus A$ is negligible.
\end{lem}
\begin{proof}
Recall that $f$ and $g$ are supposed to be bounded by a constant $\gamma$ and 
$$h(t) = |f(t,0)| +2^{p-1} \left|\int_0^t g_r \olaB\right|^p.$$
Thus for any $n$ 
$$\Prb\left( \sup_{t\in[0,T]} h(t) \geq n \right) \leq \Prb \left( \sup_{t\in[0,T]} \left|\int_0^t g_r \olaB\right|^p \geq 2^{1-p} (n-\gamma) \right).$$
The Burkholder-Davis-Gundy inequality shows that 
$$\E \left(  \sup_{t\in[0,T]} \left|\int_0^t g_r \olaB\right|^p \right) \leq C_p \gamma^p T^{p/2}.$$
And by Markov inequality for $n > \gamma$
$$\Prb \left( \sup_{t\in[0,T]} \left|\int_0^t g_r \olaB\right|^p \geq 2^{1-p} (n-\gamma) \right) \leq \frac{C_p \gamma^p T^{p/2}}{2^{1-p} (n-\gamma)}.$$

Note $\zeta=\sup_{t\in[0,T]} h(t)$. Now if $A \in \cH^n$, then we have two cases: either the set $\{\zeta < n\}$ is included either in $A$ or in $A^c$. And if $n \geq m$, then $\{\zeta < m\} \subset \{\zeta < n\}$. Hence if $A$ is in $\displaystyle  \bigvee_{n\geq m}\cH^n$, then $\{\zeta < m\} \subset A$ or $A \cap \{\zeta < m\} = \emptyset$. And thus $\Prb(A)\wedge \Prb(A^c) \leq C/(m-\gamma)$ for any $m > \gamma$. Finally if $\displaystyle A \in \cH^\infty = \bigcap_{m\in \N}  \bigvee_{n\geq m}\cH^n$, then $\Prb(A) =0$ or $\Prb(A)=1$. 
\end{proof}

From the previous lemma, if we define
$$Y_t = U_t - \int_0^t g_r \olaB$$
we obtain a solution $(Y,Z)$ to the BDSDE:
\begin{equation*} 
Y_{t} = \xi + \int_{t}^{T} f \left( r,Y_{r} \right) dr + \int_t^T g_r \overleftarrow{dB_r} - \int_{t}^{T} Z_{r}dW_{r}, \ 0 \leq t \leq T.
\end{equation*}
From the boundedness assumption on $g$ and since $(U,Z) \in \bB^{2p}(0,T)$, we have:
$$\E \left[ \sup_{t \in [0,T]} |Y_t|^{2p} +\left(  \int_0^T |Z_r|^{2p} dr \right) \right] < +\infty.$$
From the previous proof $Y_t$ is $\tri_t$-measurable. Then using the same argument as in \cite{pard:peng:94}, the process $Z_t$ is also $\tri_t$-measurable. In other words $(Y,Z)\in \cE^{2p}(0,T)$. 

Now we only assume that for some $p\geq 1$
\begin{equation}\label{eq:L_2p_cond}
\E  \left[ |\xi|^{2p} +\left(  \int_0^T (|f(t,0)|^{2} +  |g_t|^{2} ) dt \right)^p \right] < +\infty.
\end{equation}
\begin{lem}
Under Conditions \eqref{eq:L_2p_cond}, the BDSDE \eqref{eq:BDSDE_1} has a unique solution $(Y,Z) \in \cE^{2p}(0,T)$.
\end{lem}
\begin{proof}
For any $n\in \N^*$ define $\Theta_n$ by 
$$\Theta_n(y) = \left\{ \begin{array}{cl}
y & \mbox{if } |y| \leq n, \\
\displaystyle n\frac{y}{|y|} & \mbox{if } |y |> n ,
\end{array} \right.$$
and $\xi^n = \Theta_n(\xi)$, $g^n_t = \Theta_n(g_t)$ and $f^n(t,y) = f(t,y)-f(t,0)+\Theta_n(f(t,0))$. Thus for a fixed $n$, there exists a solution $(Y^n,Z^n)$ to the BDSDE \eqref{eq:BDSDE_1} with $\xi^n$, $f^n$ and $g^n$ instead of $\xi$, $f$ and $g$:
\begin{equation*} 
Y^n_{t} = \xi^n + \int_{t}^{T} f^n \left( r,Y^n_{r} \right) dr + \int_t^T g^n_r \overleftarrow{dB_r} - \int_{t}^{T} Z^n_{r}dW_{r}, \ 0 \leq t \leq T.
\end{equation*}
Define for any $n$ and $m$
$$\Delta \xi = \xi^m-\xi^n, \quad \Delta f(t,y) = f^m(t,y) -f^n(t,y),\quad  \Delta g_t = g^m_t - g^n_t,$$
and 
$$\Delta Y_t = Y^m_t - Y^n_t, \quad \Delta Z_t = Z^m_t - Z^n_t.$$

From the It\^o formula with $\al =2 \mu +1 $, we have:
\begin{eqnarray*}
&&e^{\al t} |\Delta Y_t|^2  + \int_t^T e^{\al r} |\Delta Z_r |^2 dr = e^{\al T} |\Delta \xi|^2 + 2\int_t^Te^{\al s} \Delta Y_s ( f^m(s,Y^m_s) - f^n(s,Y^n_s)) ds \\
&& \qquad -\int_t^T \al e^{\al s} |\Delta Y_s|^2 ds - 2 \int_t^T e^{\al s}\Delta Y_s \Delta Z_s dW_s - 2 \int_t^Te^{\al s}  \Delta Y_s \Delta g_s \overleftarrow{dB_s} \\
& &\qquad + \int_t^T e^{\al s}|\Delta g_s|^2 ds .
\end{eqnarray*}
From assumption \eqref{ineq:monoton} on $f$ and $2|ab|\leq a^2 + b^2$, we obtain:
\begin{eqnarray*}
&&e^{\al t} |\Delta Y_t|^2  + \int_t^T e^{\al r} |\Delta Z_r |^2 dr \leq e^{\al T} |\Delta \xi|^2 + \int_t^Te^{\al s}| \Delta f(s,0) |^2 ds \\
&& \qquad - 2 \int_t^T e^{\al s}\Delta Y_s \Delta Z_s dW_s - 2 \int_t^Te^{\al s}  \Delta Y_s \Delta g_s \overleftarrow{dB_s} + \int_t^T e^{\al s}|\Delta g_s|^2 ds .
\end{eqnarray*}
Using BDG inequality (see the proof of Lemma 3.1 in \cite{bria:dely:hu:03} for the details) we deduce that there exists a constant $C$ depending on $p$, $\mu$ and $T$ such that 
\begin{equation} \label{eq:estim_2p_DeltaZ}
\E \left[  \left( \int_0^T |\Delta Z_r |^2 dr \right)^{p} \right] \leq C\E  \left[ \sup_{t\in [0,T]}|\Delta Y_t|^{2p} + \left(  \int_0^T | \Delta f(s,0) |^{2} ds +  \int_0^T |\Delta g_s|^{2} ds \right)^p \right].
\end{equation}

Since $p \geq 1$ we can apply It\^o formula with the $C^2$-function $\theta(y) = |y|^{2p}$ to the process $\Delta Y$. Note that
$$\frac{\partial \theta}{\partial y_i} (y)= 2p y_i |y|^{2p-2}, \frac{\partial^2 \theta}{\partial y_i \partial y_j} (y)= 2p |y|^{2p-2} \delta_{i,j} + 2p(2p-2)y_iy_j|y|^{2p-4}$$
where $\delta_{i,j}$ is the Kronecker delta. Therefore for every  $0\leq t \leq T$ we have:
\begin{eqnarray} \label{eq:Ito_formula_2p}
&&e^{\al t} |\Delta Y_{t}|^{2p} + \frac{1}{2}\int_t^T e^{\al s}\tr \left( D^2\theta(\Delta Y_s) \Delta Z_s \Delta Z_s^* \right) ds = e^{\al T}|\Delta \xi|^{2p}  - \int_t^T\al e^{\al s} |\Delta Y_{s}|^{2p} ds  \\ \nonumber
&& \quad + \int_{t}^{T} 2p e^{\al s}\Delta Y_s |\Delta Y_s|^{2p-2} (f^m(s,Y^m_s) - f^n(s,Y^n_s)) ds \\  \nonumber
& &\quad   - 2 p  \int_{t}^{T}e^{\al s} \Delta Y_{s} |\Delta Y_{s}|^{2p-2} \Delta Z_s dW_s - 2 p  \int_{t}^{T} e^{\al s} \Delta Y_{s} |\Delta Y_{s}|^{2p-2} \Delta g_s \overleftarrow{dB_s}    \\ \nonumber
& &\quad  + \frac{1}{2}\int_t^T e^{\al s}\tr \left( D^2\theta(\Delta Y_s) \Delta g_s \Delta g_s^* \right) ds.
\end{eqnarray}
Remark that  
$$\tr (D^2\theta(y) z z^*)   \geq 2p |y|^{2p-2} |z|^2.$$
Moreover from assumption \eqref{ineq:monoton} on $f$ and Young's inequality 
we obtain:
\begin{eqnarray*}
&& \int_{t}^{T} 2p e^{\al s}\Delta Y_s |\Delta Y_s|^{2p-2} (f^m(s,Y^m_s) - f^n(s,Y^n_s)) ds \\
&& \leq  \int_{t}^{T} 2p \mu e^{\al s} |\Delta Y_s|^{2p} ds + \int_{t}^{T} 2p e^{\al s}\Delta Y_s |\Delta Y_s|^{2p-2}  \Delta f(s,0) ds \\
&& \leq  \int_{t}^{T} 2p \mu e^{\al s} |\Delta Y_s|^{2p} ds + \int_{t}^{T} 2p e^{\al s} |\Delta Y_s|^{2p-1} | \Delta f(s,0) | ds. 
\end{eqnarray*}
The following term 
$$\int_t^T e^{\al s}\tr \left( D^2\theta(\Delta Y_s) \Delta g_s \Delta g_s^* \right) ds$$
can be controlled by:
\begin{eqnarray*}
&&\int_t^T e^{\al s}\tr \left( D^2\theta(\Delta Y_s) \Delta g_s \Delta g_s^* \right) ds  \leq C_p \int_t^T e^{\al s}|\Delta Y_s|^{2p-2} |\Delta g_s|^2  ds
\end{eqnarray*}
Now let us come back to \eqref{eq:Ito_formula_2p}: 
\begin{eqnarray} \label{eq:Ito_formula_2p_2}
&&e^{\al t} |\Delta Y_{t}|^{2p} +p \int_{t}^{T} e^{\al s}|\Delta Y_s|^{2p-2} |\Delta Z_s|^2  ds \leq e^{\al T}|\Delta \xi|^{2p} \\ \nonumber
&& \quad + \int_{t}^{T} 2p e^{\al s} |\Delta Y_s|^{2p-1} | \Delta f(s,0) | ds + C_p \int_t^T e^{\al s}|\Delta Y_s|^{2p-2} |\Delta g_s|^2  ds  \\ \nonumber
&& \quad  +  \left( 2p\mu   -\al \right)\int_t^T e^{\al s} |\Delta Y_{s}|^{2p} ds  \\ \nonumber
& &\quad   - 2 p  \int_{t}^{T}e^{\al s} \Delta Y_{s} |\Delta Y_{s}|^{2p-2} \Delta Z_s dW_s - 2 p  \int_{t}^{T} e^{\al s} \Delta Y_{s} |\Delta Y_{s}|^{2p-2} \Delta g_s \overleftarrow{dB_s}  
\end{eqnarray}
If we note  
\begin{eqnarray*}
X & = & e^{\al T}|\Delta \xi|^{2p} + \int_{0}^{T} 2p e^{\al s} |\Delta Y_s|^{2p-1} | \Delta f(s,0) | ds + C_p \int_0^T e^{\al s}|\Delta Y_s|^{2p-2} |\Delta g_s|^2  ds \\ 
M_t & = & \int_{0}^{t}e^{\al s} \Delta Y_{s} |\Delta Y_{s}|^{2p-2} \Delta Z_s dW_s \\
N_t & = & \int_{t}^{T} e^{\al s} \Delta Y_{s} |\Delta Y_{s}|^{2p-2} \Delta g_s \overleftarrow{dB_s},
\end{eqnarray*}
by Young's inequality, $M$ and $N$ are true martingales. Hence if we choose $\al = 2p\mu+1$, taking the expectation in \eqref{eq:Ito_formula_2p_2} we obtain:
\begin{eqnarray} \label{eq:Ito_formula_2p_3}
&&\E \int_0^T e^{\al s} |\Delta Y_{s}|^{2p} ds  + p \E \int_{0}^{T} e^{\al s}|\Delta Y_s|^{2p-2} |\Delta Z_s|^2  ds \leq \E (X).
\end{eqnarray}
Then for any $\delta > 0$ 
\begin{eqnarray*}
&& \E \langle M \rangle_T^{1/2} \leq \E \left[ \left( \sup_{t\in [0,T]} e^{\al t /2} |\Delta Y_t|^{p}\right) \left( \int_{0}^{T} e^{\al s}|\Delta Y_s|^{2p-2} |\Delta Z_s|^2  ds \right)^{1/2} \right] \\
&& \quad \leq \frac{1}{2}\delta^{2} \E \left( \sup_{t\in [0,T]} e^{\al t} |\Delta Y_t|^{2p}\right)  + \frac{1}{2\delta^{2}}  \E\left(   \int_{0}^{T} e^{\al s}|\Delta Y_s|^{2p-2} |\Delta Z_s|^2  ds \right)^{p} .
\end{eqnarray*}
and
\begin{eqnarray*}
&& \E \langle N \rangle_0^{1/2} \leq \E \left[ \left( \sup_{t\in [0,T]} e^{(2p-1)\al t/(2p)} |\Delta Y_t|^{2p-1}\right) \left( \int_{0}^{T} e^{\al s/p} |\Delta g_s|^2  ds \right)^{1/2} \right] \\
&& \quad \leq \frac{2p-1}{2p} \delta^{-2p/(2p-1)} \E \left( \sup_{t\in [0,T]} e^{\al t} |\Delta Y_t|^{2p}\right)  + \frac{\delta^{2p}}{2p}  \E\left(   \int_{0}^{T} e^{\al s/p} |\Delta g_s|^2  ds \right)^{p} .
\end{eqnarray*}
Coming back to \eqref{eq:Ito_formula_2p_2} with $\al = 2p\mu+1$, and taking the supremum over $t \in [0,T]$ and then the expectation, with the BDG inequality we have:
\begin{eqnarray*} 
&&\E \left( \sup_{t\in [0,T]} e^{\al t} |\Delta Y_{t}|^{2p} \right) \leq \E (X) \\ \nonumber
& &\quad  + \frac{1}{2}\delta^{2} \E \left( \sup_{t\in [0,T]} e^{\al t} |\Delta Y_t|^{2p}\right)  + \frac{1}{2\delta^{2}}  \E\left(   \int_{0}^{T} e^{\al s}|\Delta Y_s|^{2p-2} |\Delta Z_s|^2  ds \right)^{p}    \\ \nonumber
&& \quad + \frac{2p-1}{2p} \delta^{-2p/(2p-1)} \E \left( \sup_{t\in [0,T]} e^{\al t} |\Delta Y_t|^{2p}\right)  + \frac{\delta^{2p}}{2p}  \E\left(   \int_{0}^{T} e^{\al s/p} |\Delta g_s|^2  ds \right)^{p} .
\end{eqnarray*}
We can choose $\delta$ small enough such that with \eqref{eq:Ito_formula_2p_3}, finally 
\begin{equation} \label{eq:estim_2p_DeltaY}
\E \left( \sup_{t\in [0,T]} e^{\al t} |\Delta Y_{t}|^{2p} \right) \leq C  \E \left( X  \right) + C  \E\left(   \int_{0}^{T} e^{\al s/p} |\Delta g_s|^2  ds \right)^{p}
\end{equation}
for some constant $C$ depending only on $\mu$ and $p$. Now once again with Young's inequality for any $\eps > 0$
\begin{eqnarray*} 
&& \E \int_{0}^{T} e^{\al s} |\Delta Y_s|^{2p-1} | \Delta f(s,0) | ds \leq \E \left[ \left( \sup_{t \in [0,T]} e^{\al t(2p-1)/(2p)}|\Delta Y_t|^{2p-1} \right) \int_{0}^{T} e^{\al s/(2p)} | \Delta f(s,0) | ds \right] \\
&& \qquad \leq \eps^{2p/(2p-1)} \frac{2p-1}{2p} \E \left( \sup_{t \in [0,T]} e^{\al t}|\Delta Y_t|^{2p} \right) +\frac{\eps^{-2p}}{2p} \E \left[ \left(  \int_{0}^{T} e^{\al s/p} | \Delta f(s,0) |^2 ds \right)^p \right]
\end{eqnarray*}
and 
\begin{eqnarray*} 
&& \E \int_0^T e^{\al s}|\Delta Y_s|^{2p-2} |\Delta g_s|^2  ds \leq \E \left[ \left( \sup_{t \in [0,T]} e^{\al t(p-1)/p}|\Delta Y_t|^{2p-2} \right) \int_{0}^{T} e^{\al s/p} | \Delta g_s|^2 ds \right] \\
&& \qquad \leq \eps^{p/(p-1)} \frac{p-1}{p} \E \left( \sup_{t \in [0,T]} e^{\al t}|\Delta Y_t|^{2p} \right) +\frac{\eps^{-p}}{p} \E \left[ \left(  \int_{0}^{T} e^{\al s/p} | \Delta g_s |^2 ds \right)^p \right].
\end{eqnarray*}
Using these two inequalities, \eqref{eq:estim_2p_DeltaZ} and \eqref{eq:estim_2p_DeltaY}:
\begin{eqnarray*} 
&& \E \left( \sup_{t\in [0,T]} e^{\al t} |\Delta Y_{t}|^{2p} \right) + \E \left[  \left( \int_0^T e^{\al s/p} |\Delta Z_s |^2 ds \right)^{p} \right] \\
&& \qquad \leq C\E  \left[ e^{\al T} |\Delta \xi|^{2p} + \left(  \int_0^T e^{\al s/p}| \Delta f(s,0) |^{2} ds +  \int_0^T e^{\al s/p} |\Delta g_s|^{2} ds \right)^p \right].
\end{eqnarray*}
Therefore with \eqref{eq:estim_2p_DeltaZ}, $(Y^n,Z^n)$ is a Cauchy sequence which converges to $(Y,Z)$ and the limit process $(Y,Z)\in \cE^{2p}(0,T)$ satisfies the BDSDE \eqref{eq:BDSDE_1}.
\end{proof}

\begin{rem}
Can we assume a weaker growth condition on $f$ ? Suppose that there exists a non decreasing function $\psi : \R^+ \to \R^+$ such that 
$$|f(t,y)|\leq |f(t,0)| + \psi(|y|).$$
Using the same transformation, we have to control:
$$|\phi(t,y)| = |f(t,y+\int_0^t g_r \overleftarrow{dB_r} )| \leq  |f(t,0)| + \psi(|y+\int_0^t g_r \overleftarrow{dB_r}|).$$
If it is possible to find two functions $\psi_1$ and $\psi_2$ such that $\psi(y+z) \leq \psi_1(y) + \psi_2(z)$ and if $\psi_2 (|\int_0^t g_r \overleftarrow{dB_r}|)$ belongs to $L^2(\Omega)$ for any bounded process $g_t$, it may be possible to obtain a solution with the desired properties to the BDSDE \eqref{eq:BDSDE}.

\end{rem}

\subsection{General case}

The general case can be deduced from the previous one by a fixed-point argument. We still assume that Condition \eqref{eq:L_2p_cond} holds. Let us construct the following sequence:  $(Y^0,Z^0) = (0,0)$ and for $n \in \N$ and any $0 \leq t \leq T$
\begin{equation} \label{eq:approx_gene_case}
Y^{n+1}_{t} = \xi + \int_{t}^{T} f \left( r,Y^{n+1}_{r},Z^n_r \right) dr + \int_t^T g \left( r,Y^{n}_{r},Z^{n}_{r} \right) \overleftarrow{dB_r} - \int_{t}^{T} Z^{n+1}_{r}dW_{r}.
\end{equation}
Indeed if 
$$\E \left[ \sup_{t \in [0,T]} |Y^n_t|^{2p} +\left(  \int_0^T |Z^n_r|^2 dr \right)^p \right] < +\infty$$
then from \eqref{eq:property_Lip_g} and \eqref{eq:L2_integ}, the process $g^n_r = g \left( r,Y^{n}_{r},Z^{n}_{r} \right)$ satisfies 
$$\E \left( \int_0^T |g^n_r|^2 dr \right)^p < +\infty.$$
Moreover the process $f^n(r,0) = f(r,0,Z^n_r)$ verifies
$$\E \left(  \int_0^T |f(r,0,Z^n_r)|^2 ds  \right)^p \leq C_p K_f^{2p} \E \left(\int_0^T |Z^n_r|^2 dr \right)^p + C_p \E \left(\int_0^T |f(r,0,0)|^2 dr \right)^p < +\infty.$$
The previous section shows that $(Y^{n+1},Z^{n+1})$ exists and satisfies \eqref{eq:approx_gene_case} with 
$$\E \left[ \sup_{t \in [0,T]} |Y^{n+1}_t|^{2p} +\left(  \int_0^T |Z^{n+1}_r|^2 dr \right)^p \right] < +\infty.$$
Hence the sequence of processes $(Y^{n},Z^{n})$ is well defined.

Now as before define for any $n$ and $m$
$$\Delta Y^{n}_t = Y^{n+1}_t - Y^n_t, \quad \Delta Z^{n}_t = Z^{n+1}_t - Z^n_t ,\quad \Delta g^{n}_t = g(t,Y^{n+1}_t,Z^{n+1}_t) - g(t,Y^n_t,Z^n_t) .$$
From the It\^o formula with $\al > 0 $, we have:
\begin{eqnarray*}
&&e^{\al t} |\Delta Y^{n}_t|^2  + \int_t^T e^{\al r} |\Delta Z^{n}_r |^2 dr =  2\int_t^Te^{\al s} \Delta Y^{n}_s ( f(s,Y^{n+1}_s,Z^n_s) - f(s,Y^n_s,Z^{n-1}_s)) ds \\
&& \qquad -\int_t^T \al e^{\al s} |\Delta Y^{n}_s|^2 ds - 2 \int_t^T e^{\al s}\Delta Y^{n}_s \Delta Z^{n}_s dW_s - 2 \int_t^Te^{\al s}  \Delta Y^{n}_s \Delta g^{n-1}_s \overleftarrow{dB_s} \\
& &\qquad + \int_t^T e^{\al s}|\Delta g^{n-1}_s|^2 ds .
\end{eqnarray*}
Using the Lipschitz assumption on $g$, we have
$$|\Delta g^{n-1}_s|^2 \leq K_g |\Delta Y^{n-1}_s|^2 + \eps |\Delta Z^{n-1}_s|^2.$$
And
$$\Delta Y^{n}_s ( f(s,Y^{n+1}_s,Z^n_s) - f(s,Y^n_s,Z^{n-1}_s)) \leq \mu |\Delta Y^n_s|^2 +\sqrt{K_f} |\Delta Y^n_s ||\Delta Z^{n-1}_s|.$$
Thus 
\begin{eqnarray} \label{eq:cauchy_seq_gene_case}
&&e^{\al t} |\Delta Y^{n}_t|^2  + \int_t^T e^{\al r} |\Delta Z^{n}_r |^2 dr \leq  (2\mu -\al)\int_t^Te^{\al s} |\Delta Y^n_s|^2 ds \\ \nonumber 
&& \qquad + 2 \sqrt{ K_f } \int_t^T e^{\al s} |\Delta Y^n_s ||\Delta Z^{n-1}_s| ds - 2 \int_t^T e^{\al s}\Delta Y^{n}_s \Delta Z^{n}_s dW_s  \\ \nonumber
& &\qquad- 2 \int_t^Te^{\al s}  \Delta Y^{n}_s \Delta g^{n-1}_s \overleftarrow{dB_s} + K_g \int_t^T e^{\al s} |\Delta Y^{n-1}_s|^2 ds + \eps \int_t^T e^{\al s} |\Delta Z^{n-1}_s|^2 ds .
\end{eqnarray}
Using the inequality $ab \leq \eta a^2 + \frac{1}{\eta} b^2$, we have
$$2 \sqrt{ K_f } \int_t^T e^{\al s} |\Delta Y^n_s ||\Delta Z^{n-1}_s| ds \leq \eta K_f  \int_t^T e^{\al s} |\Delta Y^n_s |^2 ds + \frac{1}{\eta} \int_t^T |\Delta Z^{n-1}_s| ds.$$
Therefore taking the expectation in \eqref{eq:cauchy_seq_gene_case} we deduce that 
\begin{eqnarray*}
&&\E e^{\al t} |\Delta Y^{n}_t|^2  +\E \int_t^T e^{\al r} |\Delta Z^{n}_r |^2 dr \leq  (2\mu + \eta K_f -\al)\E \int_t^Te^{\al s} |\Delta Y^n_s|^2 ds \\
&& \qquad + \left( \frac{1}{\eta} +\eps \right) \E \int_t^T e^{\al s}|\Delta Z^{n-1}_s| ds  + K_g\E  \int_t^T e^{\al s} |\Delta Y^{n-1}_s|^2 ds .
\end{eqnarray*}
Take $t=0$, $\eta = \frac{2}{1-\eps}$ and $\al = 2\mu + \frac{2K_f}{1-\eps} +  \frac{2K_g}{1+\eps}$ such that 
\begin{eqnarray} \label{eq:Cauchy_sequ_gene_case_2}
&&\E \int_0^T e^{\al r} |\Delta Z^{n}_r |^2 dr + \frac{2K_g}{1+\eps} \E \int_0^Te^{\al s} |\Delta Y^n_s|^2 ds \\ \nonumber
&&\qquad \leq \left( \frac{1+\eps}{2} \right) \E \int_0^T e^{\al s}|\Delta Z^{n-1}_s| ds  + \left( \frac{1+\eps}{2} \right) \E  \int_0^T \frac{2K_g}{1+\eps} e^{\al s} |\Delta Y^{n-1}_s|^2 ds .
\end{eqnarray}
Since $(1+\eps)/2 < 1$, the sequence $(Y^n,Z^n)$ is a Cauchy sequence in $L^2((0,T)\times \Omega)$ and converges to some process $(Y,Z)$. Moreover by the BDG inequality we also obtain:
\begin{eqnarray*}
&&\E  \sup_{t\in [0,T]} \left| \int_t^T e^{\al s}\Delta Y^{n}_s \Delta Z^{n}_s dW_s \right| \leq  4 \E \left(  \int_0^Te^{2\al s} | \Delta Y^{n}_s |^2 |\Delta Z^{n}_s|^2 ds \right)^{1/2} \\
& &\qquad \leq \frac{1}{4} \E \left( \sup_{t\in [0,T]}  e^{\al t}| \Delta Y^{n}_t |^2 \right) + 64 \E \int_0^T e^{\al s} |\Delta Z^{n}_s|^2 ds,
\end{eqnarray*}
and
\begin{eqnarray*}
&&\E  \sup_{t\in [0,T]} \left| \int_t^Te^{\al s}  \Delta Y^{n}_s \Delta g^{n-1}_s \overleftarrow{dB_s} \right| \leq  4 \E \left(  \int_0^Te^{2\al s} | \Delta Y^{n}_s |^2 |\Delta g^{n-1}_s|^2 ds \right)^{1/2} \\
& &\qquad \leq \frac{1}{4}  \E \left( \sup_{t\in [0,T]}  e^{\al t}| \Delta Y^{n}_t |^2 \right) + 32 \E \int_0^T e^{\al s} |\Delta g^{n-1}_s|^2 ds \\
& &\qquad \leq \frac{1}{4} \E \left( \sup_{t\in [0,T]}  e^{\al t}| \Delta Y^{n}_t |^2 \right) + 32K_g \E \int_0^T e^{\al s} |\Delta Y^{n-1}_s|^2 ds + 32\eps \E \int_0^T e^{\al s} |\Delta Z^{n-1}_s|^2 ds .
\end{eqnarray*}
Coming back to \eqref{eq:cauchy_seq_gene_case} and using \eqref{eq:Cauchy_sequ_gene_case_2} we have for some constant $C$
\begin{eqnarray*} 
&& \E \sup_{t \in [0,T]} e^{\al t} |\Delta Y^{n}_t|^2 \leq  C \E \int_0^T e^{\al r} |\Delta Z^{n-1}_r |^2 dr + C \E \int_0^Te^{\al s} |\Delta Y^{n-1}_s|^2 ds.
\end{eqnarray*}
We deduce also the convergence of $Y^n$ to $Y$ under this strong topology. Therefore $(Y,Z)$ satisfies the general BDSDE:
\begin{equation*}
Y_{t} = \xi + \int_{t}^{T} f \left( r,Y_{r},Z_r \right) dr + \int_t^T g \left( r,Y_{r},Z_{r} \right) \overleftarrow{dB_r} - \int_{t}^{T} Z_{r}dW_{r}, \ 0 \leq t \leq T.
\end{equation*}
Hence we have proved Theorem \ref{thm:existence_sol}. To obtain that $(Y,Z)\in \cE^{2p}(0,T)$ under Condition \eqref{eq:L_2p_condition}, we can just apply Theorem 1.4 in \cite{pard:peng:94} with straightforward modifications.

\subsection{Extension, comparison result}

The extension of $L^p$ solutions, $p\in (1,2)$, is done in Aman \cite{aman:12}. We just want here to recall the comparison principle for BDSDE (see \cite{shi:gu:liu:05}, \cite{lin:wu:11} or \cite{otma:mrha:09} on this topic). We will widely use the result in the next sections.
\begin{prop}\label{prop:comp_result}
Assume that BDSDE \eqref{eq:BDSDE} with data $(f^1,g,\xi^1)$ and $(f^2,g,\xi^2)$ have solutions $(Y^1,Z^1)$ and $(Y^2,Z^2)$ in $\cE^2(0,T)$, respectively. The coefficient $g$ satisfies \eqref{eq:property_Lip_g}. If $\xi^1 \leq \xi^2$, a.s., and $f^1$ satisfies Assumptions \eqref{ineq:monoton} and \eqref{ineq:lip_Z}, for all $t \in [0,T]$, $f^1(t,Y_t^2,Z_t^2) \leq f^2(t,Y_t^2,Z_t^2)$, a.s. (resp. $f^2$ satisfies \eqref{ineq:monoton} and \eqref{ineq:lip_Z}, for all $t \in [0,T]$, $f^1(t,Y_t^1,Z_t^1) \leq f^2(t,Y_t^1,Z_t^1)$, a.s.), then we have $Y_t^1 \leq Y_t^2$, a.s., for all $t\in [0,T]$.
\end{prop}
\begin{proof}
The proof is almost the same as Lemma 3.1 in \cite{lin:wu:11}. We define 
$$( \hat Y_t , \hat Z_t ) = ( Y^1_t - Y^2_t , Z^1_t - Z^2_t) , \quad \xi = \xi^1-\xi^2 ,$$
then $( \hat Y_t , \hat Z_t )$ satisfies the following BDSDE: for all $t \in [0, T ]$,
\begin{eqnarray*}
\hat Y_{t} & = & \hat \xi + \int_{t}^{T} \left[ f^1 \left( r,Y^1_{r},Z^1_r \right) - f^2 \left( r,Y^2_{r},Z^2_r \right) \right] dr \\
& & + \int_t^T \left[ g \left( r,Y^1_{r},Z^1_{r} \right) -  g \left( r,Y^2_{r},Z^2_{r} \right) \right]\overleftarrow{dB_r} - \int_{t}^{T} \hat Z_{r}dW_{r}.
\end{eqnarray*}
We apply It\^o's formula to $(\hat Y_t^+)^2$:
\begin{eqnarray*}
&& (\hat Y_t^+)^2 + \int_t^T\ind_{\hat Y_r > 0} |\hat Z_r|^2 dr \leq (\hat \xi^+)^2 +2 \int_t^T \hat Y_r^+\left[ f^1 \left( r,Y^1_{r},Z^1_r \right) - f^2 \left( r,Y^2_{r},Z^2_r \right) \right] dr \\
&&\qquad +   2 \int_t^T \hat Y_r^+\left[ g \left( r,Y^1_{r},Z^1_{r} \right) -  g \left( r,Y^2_{r},Z^2_{r} \right) \right]\overleftarrow{dB_r} \\
&&\qquad - 2 \int_t^T \hat Y_r^+ \hat Z_{r}dW_{r} + \int_t^T \ind_{\hat Y_r > 0} |g(r, Y_r^1, Z_r^1) - g(r, Y_r^2, Z_r^2)|^2dr.
\end{eqnarray*}
Now from \eqref{ineq:monoton} and \eqref{ineq:lip_Z}
\begin{eqnarray*}
&& \hat Y_r^+\left[ f^1 \left( r,Y^1_{r},Z^1_r \right) - f^2 \left( r,Y^2_{r},Z^2_r \right) \right] \\
&& = \hat Y_r^+\left[ f^1 \left( r,Y^1_{r},Z^1_r \right) - f^1 \left( r,Y^2_{r},Z^2_r \right) \right] + \hat Y_r^+\left[ f^1 \left( r,Y^2_{r},Z^2_r \right) - f^2 \left( r,Y^2_{r},Z^2_r \right) \right] \\
&& \leq \mu (\hat Y_r^+)^2 + K_f \hat Y_r^+|\hat Z_r|.
\end{eqnarray*}
The rest of the proof is exactly the same as in \cite{lin:wu:11}. Using Gronwall's Lemma, we deduce that $\E (\hat Y_t^+)^2 = 0$ for any $t\in[0,T]$.
\end{proof}

\subsection{SPDE with monotone generator}

In this section we want to extend the results of Bally and Matoussi \cite{ball:mato:01} (more precisely Theorem 3.1) to the monotone case. We will use the same notations as in \cite{ball:mato:01}. For all $(t,x) \in [0,T] \times \R^{d}$, we denote by $X^{t,x}$ as the solution of the SDE \eqref{eq:eds2} with $b \in C^2_b$ and $\sigma \in C^3_b$. We assume that Conditions (A) and \eqref{eq:cond_L_2p_SPDE} hold. $(Y^{t,x},Z^{t,x})$ is the unique solution of the BDSDE in $\cE^{2p}(0,T)$. We define 
$$u(t,x) = Y^{t,x}_t, \quad v(t,x) = Z^{t,x}_t.$$
We recall the definition of a weak solution. 
\begin{defin} \label{def:weak_sol_SPDE}
$u$ is a weak solution of the SPDE \eqref{eq:SPDE_gene} if the following conditions are satisfied.
\begin{enumerate}
\item For some $\delta > 0$
\begin{equation} \label{eq:weak_sol_int_cond} 
\sup_{s\leq T} \E\left[ ||u(s,.)||^{1+\delta}_{L^2_{\rho^{-1}}(\R^d)}\right] < \infty.
\end{equation}
\item For every test-function $\phi \in C^{\infty}(\R^d)$, $dt \otimes d\prb$ a.e.
\begin{equation} \label{eq:weak_sol_cont_cond} 
 \lim_{s\uparrow t} \int_{\R^d}u(s,x)\phi(x)dx=\int_{\R^d}u(t,x)\phi(x)dx.
\end{equation}
\item Finally $u$ satisfies for every function $\Psi \in C^{1,\infty}_c([0,T] \times \R^d;\R)$ 
\begin{eqnarray}\label{eq:weak_form_spde} 
& &\int_t^T\int_{\R^d}u(s,x)\partial_s\Psi(s,x)dxds+\int_{\R^d}u(t,x)\Psi(t,x) dx -\int_{\R^d}h(x)\Psi(T,x)dx\\ \nonumber
&&\qquad - \frac{1}{2}\int_t^T\int_{\R^d}(\sigma^*\nabla u)(s,x)(\sigma^*\nabla\Psi)(s,x)dx ds\\ \nonumber
&&\qquad  -\int_t^T\int_{\R^d}u(s,x) \diver\left(\left(b-\widetilde{A}\right)\Psi\right)(s,x)dx ds\\ \nonumber
&&= \int_t^T\int_{\R^d}\Psi(s,x)f(s,x,u(s,x),(\sigma^* \nabla u)(s,x)) dx ds\\ \nonumber
&&\qquad +\int_t^T\int_{\R^d} \Psi(s,x) g(s,x,u(s,x),(\sigma^*\nabla u)(s,x)) dx \overleftarrow{dB_s} .
\end{eqnarray}
\end{enumerate}
Here 
$$\widetilde A_i = \frac{1}{2} \sum_{j=1}^d \frac{\partial (\sigma\sigma^*)_{j,i}}{\partial x_j}.$$
\end{defin}

To prove Proposition \ref{prop:existence_sol_monotone_SPDE}, we can directly sketch the proof of Theorem 3.1 in \cite{ball:mato:01} step by step. Using the equivalence of norms we have:
\begin{eqnarray*}
&& \int_{\R^d} \int_0^T \E |u(s,x)|^{2p} \rho^{-1}(x) dx ds \leq C  \int_{\R^d} \int_0^T \E |u(s,X^{t,x}_s)|^{2p} \rho^{-1}(x) dx ds \\
&& = C  \int_{\R^d} \int_0^T \E |Y^{t,x}_s|^{2p} \rho^{-1}(x) dx ds \\
&& \leq C  \int_{\R^d} \left( \E |h(X^{t,x}_T)|^{2p} + \E \int_0^T \left( |f(s,X^{t,x}_s,0,0)|^{2p} + |g(s,X^{t,x}_s,0,0)|^{2p} \right) ds\right) \rho^{-1}(x) dx \\
&& \leq C \int_{\R^d}  \left[  |h(x)|^{2p}  + \int_0^T\left( |f(t,x,0,0)|^{2p} + |g(t,x,0,0)|^{2p} \right) dt \right] \rho^{-1}(x) dx.
\end{eqnarray*}
Then we define $H=h(x)dx$, 
$$F_s = f(s,x,u(s,x),v(s,x)),\quad G_s = g(s,x,u(s,x),v(s,x)).$$
From the Assumptions \eqref{ineq:lip_Z}, \eqref{eq:property_growth_f}, \eqref{eq:property_Lip_g}, \eqref{eq:cond_L_2p_SPDE} and \eqref{eq:2p_estim}, $H$, $F_s$ and $G_s$ are in $\cH_{0,\rho}'$ (see \cite{ball:mato:01} for a precise definition). Then we can use Theorem 2.1 in \cite{ball:mato:01}: $v=\sigma^* \nabla u$ and $u$ solves the linear SPDE associated to $H$, $F_s$ and $G_s$ (see Equation (16) in \cite{ball:mato:01}):
\begin{eqnarray*}
&& \int_t^T (u_s,\partial_s \phi(s,.)) ds - (H,\phi(T,.))+ (u_t,\phi(t,.)) - \int_t^T (u_s,\ope^*\phi(s,.))ds \\
&& \qquad = \int_t^T (F_s,\phi(s,.)) ds + \int_t^T (H_s,\phi(s,.)) \overleftarrow{dB_s}. 
\end{eqnarray*} 
$\ope^*$ is the adjoint of $\ope$. Thus $u$ is a weak solution of \eqref{eq:SPDE_gene}. Uniqueness can be proved exactly as in \cite{ball:mato:01}.

\section{Singular terminal condition, construction of a minimal solution} \label{sect:exist_min_sol}

From now on we assume that the terminal condition $\xi$ satisfies the property \eqref{cond_infinity}:
$$\prb(\xi \geq 0)=1 \quad \mbox{and} \quad \prb(\xi = + \infty) > 0.$$ 
For $q > 0$, let us consider the function $f: \R \to \R$, defined by $f(y)=-y|y|^{q}$. $f$ is continuous and monotone, i.e. satisfies Condition \eqref{ineq:monoton} with $\mu=0$: for all $(y,y') \in \R^{2}$:
$$(y-y')( f(y) - f(y') ) \leq 0.$$
Condition \eqref{eq:property_growth_f} is also satisfied with $p=q+1$. We also consider a function $g : [0,T]\times \Omega \times \R \times \R^d \to \R$ and we assume that Condition \eqref{eq:property_Lip_g} holds. 

\subsection{Approximation}

For every $n \in \N^{*}$, we introduce $\xi_{n} = \xi \wedge n$. $\xi_{n}$ belongs to $L^{2} \left( \Omega, \tri_{T}, \prb ; \R \right)$. We apply Theorem \ref{thm:existence_sol} with $\xi_{n}$ as the final data, and we build a sequence of random processes $(Y^{n},Z^{n}) \in \cE^2(0,T)$ which satisfy \eqref{eq:sing_BDSDE}. 

From Proposition \ref{prop:comp_result}, if $n \leq m$, $0 \leq \xi_{n} \leq \xi_{m} \leq m$, which implies for all $t$ in $[0,T]$, a.s., 
\begin{equation} \label{eq:bound_Y}
\Xi^0_t \leq Y^{n}_{t} \leq Y^{m}_{t} \leq \Xi^m_t .
\end{equation}
Here $\Xi^k$ is the first component of the unique solution $(\Xi^k,\Theta^k)$ in $\cE^2(0,T)$ of \eqref{eq:sing_BDSDE} with a deterministic terminal condition $k$. In order to have explicit and useful bound on $Y^m$ we will assume that $g(t,y,0)=0$ for any $(t,y)$ a.s. In this case for $m\geq 1$,
$$\Xi^0_t = 0, \qquad \Xi^m_t = \left( \frac{1}{q(T-t)+\frac{1}{m^{q}}} \right)^{\frac{1}{q}}.$$

We define the progressively measurable $\R$-valued process $Y$, as the increasing limit of the sequence $(Y^{n}_{t})_{n \geq 1}$:
\begin{equation} \label{eq:def_sol}
\forall t \in [0,T], \ Y_{t} = \lim_{n \rightarrow + \infty} Y^{n}_{t}.
\end{equation}
Then we obtain 
\begin{equation} \label{ineq:estimapriori}
\forall \ 0 \leq t \leq T, \quad 0 \leq Y_{t} \leq \left( \frac{1}{q(T-t)} \right)^{\frac{1}{q}}.
\end{equation}
In particular $Y$ is finite on the interval $[0,T[$ and bounded on $[0,T-\delta]$ for all $\delta > 0$. 

Here we will prove the first part of Theorem \ref{thm:exist_min_sol}, that is $(Y,Z)$ satisfies properties (D1) and (D2) of the definition \ref{definsolution}. Moreover we will obtain that there exists a constant $\kappa$, depending on $g$, s.t. for all $t \in [0,T[$,
\begin{equation} \label{ineq:estim_Z} 
\E \int_{0}^{t} | Z_{r} |^{2} dr \leq \frac{\kappa}{\left(q(T-t)\right)^{\frac{2}{q}}}, 
\end{equation}
\begin{proof}
Let $\delta > 0$ and $s \in [0,T-\delta]$. For all $0 \leq t \leq s$, It\^{o}'s formula leads to the equality:
\begin{eqnarray} \nonumber
&& | Y^{n}_{t} - Y^{m}_{t}|^{2} + \int_{t}^{s} |Z^{n}_{r}-Z^{m}_{r} |^{2} dr = | Y^{n}_{s} - Y^{m}_{s}|^{2} - 2 \int_{t}^{s} (Y^{n}_{r} -Y^{m}_{r})(Z^{n}_{r}-Z^{m}_{r})dW_{r} \\ \nonumber 
& &\qquad  + 2 \int_{t}^{s} (Y^{n}_{r} -Y^{m}_{r}) \left( f(Y^{n}_{r})-f(Y^{m}_{r}) \right) dr \\\nonumber
&& \qquad  + 2 \int_t^s (Y^{n}_{r} -Y^{m}_{r})(g(r,Y^n_r,Z^{n}_{r})-g(r,Y^m_r,Z^{m}_{r})) \overleftarrow{dB_r} \\ \nonumber
&& \qquad + \int_t^s |g(r,Y^n_r,Z^{n}_{r})-g(r,Y^m_r,Z^{m}_{r})|^2 dr \\ \nonumber
&& \quad  \leq | Y^{n}_{s}- Y^{m}_{s}|^{2} + K \int_t^s |Y^n_r - Y^m_r|^2 dr   + \eps \int_t^s |Z^n_r - Z^m_r|^2 dr \\ \nonumber
&& \qquad  - 2 \int_{t}^{s} (Y^{n}_{r} -Y^{m}_{r})(Z^{n}_{r}-Z^{m}_{r})dW_{r} \\ \nonumber
&& \qquad + 2 \int_t^s (Y^{n}_{r} -Y^{m}_{r})(g(r,Y^n_r,Z^{n}_{r})-g(r,Y^m_r,Z^{m}_{r})) \overleftarrow{dB_r} 
\end{eqnarray}
from the monotonicity of $f$ (Inequality \eqref{ineq:monoton}) and the Lipschtiz property of $g$ (Inequality \eqref{eq:property_Lip_g}).  From the properties \eqref{eq:property_Lip_g} and since $(Y,Z) \in \cE^2$, we have:
$$\E \left( \int_{t}^{s} (Y^{n}_{r} - Y^{m}_{r})(Z^{n}_{r}-Z^{m}_{r})dW_{r} \right) = 0,$$
$$\E \left( \int_{t}^{s} (Y^{n}_{r} -Y^{m}_{r})(g(r,Y^n_r,Z^{n}_{r})-g(r,Y^m_r,Z^{m}_{r})) \overleftarrow{dB_r} \right) = 0.$$
From the Burkholder-Davis-Gundy inequality, we deduce the existence of a universal constant $C$ with:
\begin{equation} \label{ineq:conv}
\E \left( \sup_{0\leq t \leq s} | Y^{n}_{t} - Y^{m}_{t}|^{2} + \int_{0}^{s} |Z^{n}_{r}-Z^{m}_{r} |^{2} dr \right) \leq C \ \E \left( | Y^{n}_{s} - Y^{m}_{s}|^{2} \right).
\end{equation}
From the estimate \eqref{ineq:estimapriori}, for $s \leq T-\delta$, $Y^{n}_{s} \leq \frac{1}{(q \delta)^{1/q}}$ and $Y_{s} \leq \frac{1}{(q \delta)^{1/q}}$.
Since $Y^{n}_{s}$ converges to $Y_{s}$ a.s., the dominated convergence theorem and the previous inequality 
\eqref{ineq:conv} imply:
\begin{enumerate}
\item for all $\delta > 0$, $(Z^{n})_{n \geq 1}$ is a Cauchy sequence
in $L^{2}(\Omega \times [0,T-\delta];\R^{d})$, and converges to $Z \in L^{2}(\Omega \times [0,T-\delta];\R^{d})$, 
\item $(Y^{n})_{n \geq 1}$ converges to
$Y$ uniformly in mean-square on the interval $[0,T-\delta]$, in particular $Y$ is continuous on $[0,T)$,
\item $(Y,Z)$ satisfies Equation \eqref{eq:sing_BDSDE} on $[0,T)$. 
\end{enumerate}
Since $Y_{t}$ is smaller than $1/(q(T-t))^{1/q}$ by \eqref{ineq:estimapriori}, and since $Z \in L^{2}(\Omega \times [0,T-\delta];\R^{d})$, applying the It\^{o} formula to $|Y|^{2}$, with $s < T$ and $0 \leq t \leq s$, we obtain:
\begin{eqnarray*} 
|Y_{t}|^{2} + \int_{t}^{s} |Z_{r} |^{2} dr  & = &|Y_{s}|^{2} - 2 \int_{t}^{s} Y_{r}Z_{r}dW_{r} + 2 \int_{t}^{s} Y_{r}  f(Y_{r}) dr \\  
&& + 2 \int_{t}^{s} Y_{r}g(r,Y_r,Z_{r}) \overleftarrow{dB_{r}} + \int_t^s |g(r,Y_r,Z_r) |^2 dr \\
& \leq & \frac{1}{\left( q(T-s) \right)^{\frac{2}{q}}} - 2 \int_{t}^{s} Y_{r} Z_{r} dW_{r}+ 2 \int_{t}^{s} Y_{r}g(r,Y_r,Z_{r}) \overleftarrow{dB_{r}} \\
&& + K \int_t^s |Y_r |^2 dr + \eps \int_t^s |Z_r |^2 dr,
\end{eqnarray*}
again thanks to Inequalities \eqref{ineq:monoton} and \eqref{eq:property_Lip_g}. From \eqref{ineq:estimapriori}, since $Z \in L^{2}([0,s] \times \Omega)$, we have: 
$$ \E \int_{t}^{s} Y_{r} Z_{r} dW_{r} = \E  \int_{t}^{s} Y_{r}g(r,Y_r,Z_{r}) \overleftarrow{dB_{r}} = 0.$$
Therefore, we deduce that there exists a constant $\kappa$ depending on $T$, $K$ and $\eps$ such that : 
\begin{equation*} 
\E \int_{0}^{s} | Z_{r} | ^{2} dr \leq \frac{\kappa}{\left(q(T-s)\right)^{\frac{2}{q}}}.
\end{equation*}
\end{proof}

Remark that if $g$ is equal to zero, $\kappa$ is equal to one. And in general 
$$\kappa = \frac{1}{1-\eps}(1 + K T).$$ 
We want to establish the following statement which completes Inequality \eqref{ineq:estim_Z}.
\begin{prop} \label{prop:sharp_estim_Z}
The next inequality is a sharper estimation on $Z$:
\begin{equation} \label{ineq:sharp_estim_Z}
\E \int_{0}^{T} (T-s)^{2/q} | Z_{s} |^{2} ds \leq \frac{8+KT}{1-\eps} \left( \frac{1}{q} \right)^{2/q}.
\end{equation}
The constants $K$ and $\eps$ are given by the assumption \eqref{eq:property_Lip_g}.
\end{prop}
\begin{proof}
First suppose there exists a constant $\alpha > 0$ such that $\prb$-a.s. $\xi \geq \alpha$. In this case, by comparison, for all integer $n$ and all $t \in [0,T]$:
$$Y^{n}_{t} \geq \left( \frac{1}{qT+1/\al^{q}} \right)^{1/q} > 0.$$
Let $\delta > 0$ and $\theta : \R \rightarrow \R$, $\theta_{q} : \R \rightarrow \R$ defined by:
\begin{displaymath}
\left\{ \begin{array}{lcll}
\theta(x) & = & \sqrt{x} & \mbox{on} \ [\delta,+\infty[,\\
\theta(x) & = & 0 & \mbox{on} \ ]-\infty,0],\\
\end{array}
\right.
\quad \mbox{and} \quad
\left\{ \begin{array}{lcll}
\theta_{q}(x) & = & x^{\frac{1}{2q}} & \mbox{on} \ [\delta,+\infty[,\\
\theta_{q}(x) & = & 0 & \mbox{on} \ ]-\infty,0],\\
\end{array}
\right.
\end{displaymath}
and such that $\theta$ and $\theta_{q}$ are non-negative, non-decreasing and in respectively $C^{2}(\R)$ and $C^{1}(\R)$. We apply the It\^{o} formula on $[0,T-\delta]$ to the function $\theta_{q}(T-t) \theta (Y^{n}_{t})$, with $\delta < (qT+1/\al^{q})^{- 1/q}$:
\begin{eqnarray*}
\theta_{q}(\delta) \theta (Y^{n}_{T-\delta}) - \theta_{q}(T) \theta (Y^{n}_{0}) & = & \frac{1}{2} \int_{0}^{T-\delta} (T-s)^{1/2q} (Y^{n}_{s})^{1/2} \left( (Y^{n}_{s})^{q} - \frac{1}{q(T-s)} \right) ds \\
& + & \frac{1}{2} \int_{0}^{T-\delta} (T-s)^{1/2q} (Y^{n}_{s})^{-1/2} Z^{n}_{s}dW_{s} \\
& - & \frac{1}{2} \int_{0}^{T-\delta} (T-s)^{1/2q} (Y^{n}_{s})^{-1/2} g(s,Y^n_s,Z^{n}_{s})\overleftarrow{dB_{s}}  \\
& - & \frac{1}{8} \int_{0}^{T-\delta} (T-s)^{1/2q} \frac{|Z^{n}_{s}|^{2} - g(s,Y^n_s,Z^{n}_{s})^2}{(Y^{n}_{s})^{3/2}} ds.
\end{eqnarray*}
If we define
$$\Psi^n_s = \frac{|Z^{n}_{s}|^{2} - g(s,Y^n_s,Z^{n}_{s})^2}{(Y^{n}_{s})^{3/2}},$$
we have
\begin{eqnarray*}
\frac{1}{8} \int_{0}^{T-\delta} (T-s)^{1/2q} \Psi^n_s ds & \leq & T^{1/2q} \theta (Y^{n}_{0}) + \frac{1}{2} \int_{0}^{T-\delta} (T-s)^{1/2q} (Y^{n}_{s})^{-1/2} Z^{n}_{s}dW_{s} \\
& & +  \frac{1}{2} \int_{0}^{T-\delta} (T-s)^{1/2q} (Y^{n}_{s})^{1/2} \left( (Y^{n}_{s})^{q} -\frac{1}{q(T-s)} \right) ds \\
&& - \frac{1}{2} \int_{0}^{T-\delta} (T-s)^{1/2q} (Y^{n}_{s})^{-1/2} g(s,Y^n_s,Z^{n}_{s})\overleftarrow{dB_{s}} 
\end{eqnarray*}
and since $Y^{n}_{s} \leq 1/(q(T-s))^{1/q}$ and $T^{1/q}Y_{0}^{n} \leq q^{-1/q}$, taking the expectation we obtain:
\begin{equation*}
\frac{1}{8}  \E \int_{0}^{T-\delta} (T-s)^{1/2q} \Psi^n_s ds  \leq  \theta_{q}(T) \theta (Y^{n}_{0}) \leq (1/q)^{1/2q},
\end{equation*}
that is for all $n$ and all $\delta > 0$ :
\begin{equation*}
\E \int_{0}^{T-\delta} (T-s)^{1/2q} \Psi^n_s ds  \leq  8 (1/q)^{1/2q}.
\end{equation*}
Using the assumption \eqref{eq:property_Lip_g} on $g$, we have
$$ (1-\eps) \frac{|Z^{n}_{s}|^{2} }{(Y^{n}_{s})^{3/2}} \leq \Psi^n_s + K (Y^n_s)^{1/2},$$
which implies 
\begin{eqnarray*}
&& (1-\eps) \E \int_{0}^{T-\delta} (T-s)^{1/2q}   \frac{|Z^{n}_{s}|^{2} }{(Y^{n}_{s})^{3/2}} ds  \leq  8 (1/q)^{1/2q} + K \E \int_{0}^{T-\delta} (T-s)^{1/2q} (Y^n_s)^{1/2}ds \\
&& \qquad  \leq  8 (1/q)^{1/2q} + K \E \int_{0}^{T-\delta}  (T-s)^{1/2q} \left( \frac{1}{q(T-s)}\right)^{1/(2q)} ds \leq (1/q)^{1/2q} (8 + KT). 
\end{eqnarray*}
Now, since $1/Y^{n}_{s} \geq (q(T-s))^{1/q}$, letting $\delta \to 0$ and with the Fatou lemma, we deduce that
$$\E \int_{0}^{T} (T-s)^{2/q} |Z_{s}|^{2} ds  \leq  \frac{8+KT}{1-\eps} (1/q)^{2/q}.$$

Now we come back to the case $\xi \geq 0$. We can not apply the It\^{o} formula because we do not have any positive lower bound for $Y^{n}$. We will approach $Y^{n}$ in the following way. We define for $n \geq 1$ and $m \geq 1$, $\xi^{n,m}$ by:
$$\xi^{n,m} = \left( \xi \wedge n \right) \vee \frac{1}{m}.$$
This random variable is in $L^{2}$ and is greater or equal to $1/m$ a.s. The BSDE \eqref{eq:sing_BDSDE}, with $\xi^{n,m}$ as terminal condition, has a unique solution $(\widetilde{Y}^{n,m},\widetilde{Z}^{n,m})$. It is immediate that if $m \leq m'$ and $n \leq n'$ then:
$$\widetilde{Y}^{n,m'} \leq \widetilde{Y}^{n',m}.$$
As for the sequence $Y^{n}$, we can define $\widetilde{Y}^{m}$ as the limit when $n$ grows to $+\infty$ of $\widetilde{Y}^{n,m}$. This limit $\widetilde{Y}^{m}$ is greater than $Y = \lim_{n \rightarrow + \infty} Y^{n}$. But for any $m$ and $n$, for $t \in [0,T]$:
\begin{eqnarray} \nonumber
&& \left| \widetilde{Y}^{n,m}_{t} -Y^{n}_{t} \right|^2 = \left| \xi^{n,m} - \xi^{n} \right|^2 - 2  \int_{t}^{T} \left[ \widetilde{Y}^{n,m}_{r} -Y^{n}_{r} \right] \left[ \left( \widetilde{Y}^{n,m}_{r} \right)^{q+1} - \left( Y^{n}_{r} \right)^{q+1} \right] dr \\ \nonumber
& &\qquad -  2 \int_{t}^{T}\left[ \widetilde{Y}^{n,m}_{r} -Y^{n}_{r} \right]   \left[ \widetilde{Z}^{n,m}_{r} - Z^{n}_{r} \right] dW_{r} - \int_{t}^{T}  \left| \widetilde{Z}^{n,m}_{r} - Z^{n}_{r} \right|^2 dr \\ \nonumber
&& \qquad + 2 \int_t^T \left[ \widetilde{Y}^{n,m}_{r} -Y^{n}_{r} \right]  \left[  g(r,\widetilde{Y}^{n,m}_r,\widetilde{Z}^{n,m}_r) -g(r,Y^{n}_r,Z^{n}_r) \right] \overleftarrow{dB_{r}}  \\ \nonumber
&& \qquad + \int_t^T  \left[  g(r,\widetilde{Y}^{n,m}_r,\widetilde{Z}^{n,m}_r) -g(r,Y^{n}_r,Z^{n}_r) \right]^2 dr \\ \nonumber
& &\leq \left| \xi^{n,m} - \xi^{n} \right|^2 - 2 \int_{t}^{T}\left[ \widetilde{Y}^{n,m}_{r} -Y^{n}_{r} \right]   \left[ \widetilde{Z}^{n,m}_{r} - Z^{n}_{r} \right] dW_{r}  \\ \nonumber
&&\qquad + 2 \int_t^T \left[ \widetilde{Y}^{n,m}_{r} -Y^{n}_{r} \right]  \left[  g(r,\widetilde{Y}^{n,m}_r,\widetilde{Z}^{n,m}_r) -g(r,Y^{n}_r,Z^{n}_r) \right] \overleftarrow{dB_{r}} \\ \label{eq:Ito_formula}
&& \qquad -(1-\eps) \int_{t}^{T}  \left| \widetilde{Z}^{n,m}_{r} - Z^{n}_{r} \right|^2 dr + (1+K)  \int_{t}^{T}  \left[ \widetilde{Y}^{n,m}_{r} - Y^{n}_{r} \right]^2 dr
\end{eqnarray}
and taking the expectation:
\begin{equation*}
\E \left| \widetilde{Y}^{n,m}_{t} -Y^{n}_{t} \right|^2  \leq \E \left| \xi^{n,m} - \xi^{n} \right|^2 + (1+K)  \int_{t}^{T}  \E \left[ \widetilde{Y}^{n,m}_{r} - Y^{n}_{r} \right]^2 dr.
\end{equation*}
Gronwall lemma shows that for any $t\in [0,T]$:
\begin{equation} \label{eq:approx_inf_xi}
\E \left| \widetilde{Y}^{n,m}_{t} -Y^{n}_{t} \right|^2  \leq e^{(1+K)T} \E \left| \xi^{n,m} - \xi^{n} \right|^2 \leq e^{(1+K)T}  \frac{1}{m^2}.
\end{equation}

To conclude we fix $\delta > 0$ and we apply the It\^{o} formula to the process $\displaystyle (T-.)^{2/q} \left| \widetilde{Y}^{n,m} - Y^{n} \right|^{2}$. This leads to the inequality: 
\begin{eqnarray*}
(1-\eps) \E \int_{0}^{T-\delta} (T-r)^{2/q} \left| \widetilde{Z}^{n,m}_{r} - Z^{n}_{r} \right|^{2} dr & \leq & \frac{2}{q} \E \int_{0}^{T-\delta} (T-s)^{(2/q)-1} \left| \widetilde{Y}^{n,m}_{s} - Y^{n}_{s} \right|^{2} ds \\
& & + (\delta)^{2/q} \E \left| \widetilde{Y}^{n,m}_{T-\delta} - Y^{n}_{T-\delta} \right|^{2} \\
& & + K \E \int_{0}^{T-\delta} (T-r)^{2/q} \left( \widetilde{Y}^{n,m}_{r} - Y^{n}_{r} \right)^{2} dr. 
\end{eqnarray*}
Let $\delta$ go to 0 in the previous inequality. We can do that because $(T-.)^{(2/q) -1}$ is integrable on the interval $[0,T]$ and because of \eqref{eq:approx_inf_xi}. Finally we have
\begin{eqnarray*}
(1-\eps) \E \int_{0}^{T} (T-r)^{2/q} \left| \widetilde{Z}^{n,m}_{r} - Z^{n}_{r} \right|^{2} dr & \leq  &\frac{e^{(1+K)T} }{m^2}\left[ \frac{2}{q} \int_{0}^{T} (T-s)^{(2/q) -1} ds + K \int_{0}^{T} (T-s)^{(2/q)} ds \right] \\
 & = & \frac{T^{2/q}e^{(1+K)T} }{m^{2}} \left( 1 + \frac{KT}{1+2/q} \right).
\end{eqnarray*}
Therefore, for all $\eta > 0:$
\begin{eqnarray*}
\E \int_{0}^{T} (T-r)^{2/q} \left| Z^{n}_{r} \right|^{2} dr & \leq & (1+ \eta) \E \int_{0}^{T} (T-r)^{2/q} \left| \widetilde{Z}^{n,m}_{r} \right|^{2} dr \\
& + &  (1 + \frac{1}{\eta}) \E \int_{0}^{T} (T-r)^{2/q} \left| \widetilde{Z}^{n,m}_{r} - Z^{n}_{r} \right|^{2} dr \\
& \leq & (1+\eta) \frac{8+KT}{1-\eps} (1/q)^{2/q} + (1 + \frac{1}{\eta})\frac{T^{2/q}e^{(1+K)T} }{m^{2}(1-\eps)} \left( 1 + \frac{KT}{1+2/q} \right).
\end{eqnarray*}
We have applied the previous result to $\widetilde{Z}^{n,m}$. Now we let first $m$ go to $+ \infty$ and then $\eta$ go to 0,
we have:
\begin{equation*} 
\E \int_{0}^{T} (T-r)^{2/q} \left| Z^{n}_{r} \right|^{2} dr  \leq \frac{8+KT}{1-\eps} (1/q)^{2/q}.
\end{equation*}
The result follows by letting finally $n$ go to $\infty$ and this achieves the proof of the proposition.
\end{proof}

\subsection{Existence of a limit at time $T$}

From now, the process $Y$ is continuous on $[0,T[$ and we define $Y_{T} = \xi$. The main difficulty will be to prove the continuity at time $T$. It is easy to show that:
\begin{equation} \label{liminf}
\xi \leq \liminf_{t \to T} Y_{t}.
\end{equation}
Indeed, for all $n \geq 1$ and all $t \in [0,T]$, $Y^{n}_{t} \leq Y_{t}$, therefore:
$$\xi \wedge n = \liminf_{t \to T} Y^{n}_{t} \leq \liminf_{t \to T} Y_{t}.$$
Thus, $Y$ is lower semi-continuous on $[0,T]$ (this is clear since $Y$ is the supremum of continuous functions). But now we will show that $Y$ has a limit on the left at time $T$. We will distinguich the case when $\xi$ is greater than a positive constant from the case $\xi$ non-negative. This will complete the proof of Theorem \ref{thm:exist_min_sol}.

\subsubsection{The case $\xi$ bounded away from zero.}

We can show that $Y$ has a limit on the left at $T$ by using It\^{o}'s formula applied to the process $1/(Y^{n})^{q}$. Suppose there exists a real $\alpha > 0$ such that $\xi \geq \alpha > 0$, $\prb$-a.s. Then from Proposition \ref{prop:comp_result} (and since $g(t,y,0)=0$), for every $n \in \N^{*}$ and every $0 \leq t \leq T$:
$$n \geq Y^{n}_{t} \geq \left( \frac{1}{q(T-t) + 1/\alpha ^{q}} \right)^{1/q} \geq \left( \frac{1}{qT + 1/\alpha ^{q}}  \right)^{1/q} > 0.$$
By the It\^{o} formula 
\begin{eqnarray}\nonumber
\frac{1}{(Y^{n}_{t})^{q}} & =  & \frac{1}{(\xi \wedge n)^{q}} + q(T - t) - \frac{q(q+1)}{2} \int_{t}^{T} \frac{\|Z^{n}_{s} \|^{2}}{\left( Y^{n}_{s} \right)^{q+2}} ds + \int_{t}^{T} \frac{q Z^{n}_{s}}{\left( Y^{n}_{s} \right)^{q+1}} d W_{s} \\ \nonumber
&& - q \int_t^T \frac{g(s,Y^n_s,Z^n_s)}{(Y^n_s)^{1+q}} \overleftarrow{dB_{s}} + \frac{q(q+1)}{2}\int_t^T \frac{(g(s,Y^n_s,Z^n_s))^2}{(Y^n_s)^{2+q}}  ds \\ \label{eq:resol_BSDE}
& = & \E^{\tri_{t}} \left(\frac{1}{(\xi \wedge n)^{q} } \right) + q(T -t)  + \int_{t}^{T} \frac{q Z^{n}_{s}}{\left( Y^{n}_{s} \right)^{q+1}} d W_{s}  \\ \nonumber
&&  - q   \int_{t}^{T}  \frac{g(s,Y^n_s,Z^n_s)}{(Y^n_s)^{1+q}} \overleftarrow{dB_{s}} -\frac{q(q+1)}{2} \int_{t}^{T} \Psi^n_s ds ,
\end{eqnarray}
where 
$$\Psi^n_s = \frac{|Z^{n}_{s} |^{2} - (g(s,Y^n_s,Z^n_s))^2 }{\left( Y^{n}_{s} \right)^{q+2}} .$$
Or for any $t \in [0,T]$:
\begin{eqnarray*}
\E \frac{1}{(Y^{n}_{t})^{q}} & = & \E \left(\frac{1}{(\xi \wedge n)^{q} } \right) + q(T -t) - \frac{q(q+1)}{2} \E \int_{t}^{T} \Psi^n_s ds.
\end{eqnarray*}
This shows that 
\begin{equation} \label{eq:bound_L2}
\sup_{n\in \N}  \E \int_{0}^{T} \Psi^n_s ds < +\infty.
\end{equation}
From the assumption on $g$, we have
\begin{equation}\label{eq:estim_1}
(1-\eps) \frac{|Z^{n}_{s} |^{2} }{\left( Y^{n}_{s} \right)^{q+2}} - K  \frac{1 }{\left( Y^{n}_{s} \right)^{q}}\leq  \Psi^n_s \Rightarrow (1-\eps) \frac{|Z^{n}_{s} |^{2} }{\left( Y^{n}_{s} \right)^{q+2}} \leq   \Psi^n_s + K  \frac{1 }{\left( Y^{n}_{s} \right)^{q}} .
\end{equation}
Form this inequality and Inequality \eqref{eq:bound_L2} we deduce that 
\begin{equation*}
\sup_{n\in \N}  \E \int_{0}^{T} \left( \frac{|Z^{n}_{s} |^{2} }{\left( Y^{n}_{s} \right)^{q+2}} + \frac{(g(s,Y^n_s,Z^n_s))^2}{(Y^n_s)^{2+q}}  \right) ds < +\infty.
\end{equation*}
Hence the two sequences 
$$\int_{t}^{T} \frac{q Z^{n}_{s}}{\left( Y^{n}_{s} \right)^{q+1}} d W_{s} \ \mbox{and } \int_t^T \frac{g(s,Y^n_s,Z^n_s)}{(Y^n_s)^{1+q}} \overleftarrow{dB_{s}} $$
converge weakly in $L^2$ to some stochastic integrals (the proof is classical and uses Mazur's lemma (see \cite{yosi:80}, chapter V.1, Theorem 2,  for example)):
$$\int_{t}^{T} U_s d W_{s} \ \mbox{and } \int_t^T V_s \overleftarrow{dB_{s}}.$$

Now we decompose $\Psi^n$ as follows
$$\Psi^n_s = (\Psi^n_s)^+ - (\Psi^n_s)^-,$$
where $x^+$ (resp. $x^-$) denotes the positive (resp. negative) part of $x$. Again from Inequality \eqref{eq:estim_1} we deduce that
$$(1-\eps) \frac{|Z^{n}_{s} |^{2} }{\left( Y^{n}_{s} \right)^{q+2}} - K  \frac{1 }{\left( Y^{n}_{s} \right)^{q}}\leq  \Psi^n_s  = (\Psi^n_s)^+ - (\Psi^n_s)^- \leq \frac{|Z^{n}_{s} |^{2}  }{\left( Y^{n}_{s} \right)^{q+2}},$$
and therefore
$$0 \leq (\Psi^n_s)^-  \leq K  \frac{1 }{\left( Y^{n}_{s} \right)^{q}} \leq K(qT + \frac{1}{\al^q}).$$
Therefore for any $t \in [0,T]$:
\begin{equation}\label{eq:estim_2}
0 \leq \frac{q(q+1)}{2} \int_{t}^{T}(\Psi^n_s)^- ds \leq \frac{q(q+1)}{2}K(qT + \frac{1}{\al^q}) (T-t).
\end{equation}
If we define 
$$\Gamma_t = \liminf_{n\to + \infty} \frac{q(q+1)}{2} \E^{\cG_{t}} \int_{t}^{T}(\Psi^n_s)^- ds,$$
we obtain that $\Gamma$ is a non negative bounded process. Since $(\Psi^n)^-$ is non negative, it is straightforward that $\Gamma$ is a supermartingale. Moreover the dominated convergence theorem proves that $\Gamma$ is a continuous process such that: $\displaystyle \lim_{t \to T} \Gamma_t = 0.$

Coming back to\eqref{eq:resol_BSDE} and taking the conditional expectation, we have
\begin{eqnarray*}
\frac{q(q+1)}{2}\E^{\cG_t} \int_{t}^{T} \Psi^n_s ds & = & \E^{\cG_t}\left( \frac{1}{(\xi \wedge n)^{q} } \right) -  \frac{1}{(Y^{n}_{t})^{q}} + q(T -t)  \\ 
&&  + q  \E^{\cG_t} \int_{t}^{T}  \frac{g(s,Y^n_s,Z^n_s)}{(Y^n_s)^{1+q}} \overleftarrow{dB_{s}} ,
\end{eqnarray*}
and the right hand side converges weakly in $L^2$. Therefore if we define
$$\Theta_t =\limsup_{n\to +\infty} \frac{q(q+1)}{2}\E^{\cG_t} \int_{t}^{T} (\Psi^n_s)^+ ds,$$
taking the weak limit, we obtain 
\begin{equation*}
\Theta_t - \Gamma_t = \E^{\cG_t}\left(  \frac{1}{\xi^{q} } \right)-  \frac{1}{(Y_{t})^{q}} + q(T -t)     + q \E^{\cG_t}  \int_{t}^{T} V_s \overleftarrow{dB_{s}}.
\end{equation*}
We can remark that $\left( \Theta_{t} \right)_{0 \leq t < T}$ is a non-negative supermartingale and for any $t\in[0,T]$:
$$\frac{1}{(Y_{t})^{q}} = q(T-t) + \E^{\cG_{t}} \left(\frac{1}{\xi^{q}} \right)  - \Theta_t + \Gamma_t + q \E^{\cG_t}  \int_{t}^{T} V_s \overleftarrow{dB_{s}}.$$
$\Theta$ being a right-continuous non-negative supermartingale, the limit of $\Theta_{t}$ as $t$ goes to $T$ exists $\prb$-a.s. and this limit $\Theta_{T^-}$ is finite $\prb$-a.s. The same holds for the backward It\^o integral with limit $M_T=0$. The $L^{1}$-bounded martingale $\E^{\cG_{t}} \left(\frac{1}{\xi^{q}} \right)$ converges a.s. to $1/ \xi^{q}$, as $t$ goes to $T$,  then the limit of $Y_{t}$ as $t \to T$ exists and is equal to:
$$\lim_{t \rightarrow T, \ t < T} Y_{t} = \frac{1}{(\frac{1}{\xi^{q}} - \Theta_{T^-} )^{1/q}}.$$
If we were able to prove that $\Theta_{T^-}$ is zero a.s., we would have shown that $Y_{T} = \xi$. 

\subsubsection{The case $\xi$ non negative}

Now we just assume that $\xi \geq 0$. We cannot apply the It\^{o} formula to $1/(Y^{n})^{q}$ because we have no positive lower bound for $Y^{n}$. We define for $n \geq 1$ and $m \geq 1$, $\xi^{n,m}$ by:
$$\xi^{n,m} = \left( \xi \wedge n \right) \vee \frac{1}{m},$$
This random variable is in $L^{2}$ and is greater or equal to $1/m$ a.s. , with $\xi^{n,m}$ as terminal condition, has a unique solution $(\widetilde{Y}^{n,m},\widetilde{Z}^{n,m})$ of our BSDE \eqref{eq:sing_BDSDE}. Let us come back to \eqref{eq:approx_inf_xi}. We have already proved that
\begin{equation*}
\E \left| \widetilde{Y}^{n,m}_{t} -Y^{n}_{t} \right|^2  \leq e^{(1+K)T} \E \left| \xi^{n,m} - \xi^{n} \right|^2 \leq e^{(1+K)T}  \frac{1}{m^2}.
\end{equation*}
Now using \eqref{eq:Ito_formula} with $t=0$ and taking the expectation we obtain first:
\begin{eqnarray*}
&&(1-\eps) \E \int_{0}^{T}  \left| \widetilde{Z}^{n,m}_{r} - Z^{n}_{r} \right|^2 dr  \leq \left( 1+(1+K)T e^{(1+K)T} \right) \frac{1}{m^2}
\end{eqnarray*}
from which we deduce that the two stochastic integrals in \eqref{eq:Ito_formula} are true martingales. Therefore we can use Burckholder-Davis-Gundy inequality and once again from \eqref{eq:Ito_formula} there exists a constant $C$ such that:
\begin{equation*}
\E \left( \sup_{t \in [0,T]} \left| \widetilde{Y}^{n,m}_{t} - Y^n_{t} \right|^2\right)  \leq \frac{C}{m^2}.
\end{equation*}
From Fatou's lemma the same inequality holds for $Y^m-Y$. Since $\widetilde{Y}^{m}$ has a limit on the left at $T$, so does $Y$. 

\subsection{Minimal solution}

In this section we will achieve the proof of Theorem \ref{thm:exist_min_sol}. Let $(\widetilde Y,\widetilde Z)$ be another non negative solution of BDSDE \eqref{eq:sing_BDSDE} in the sense of Definition \ref{definsolution}. Note that we will only use that 
$$\liminf_{t\to T} \widetilde Y_t \geq \xi$$ 
(and not the stronger condition (D3)). Then a.s. for any $t\in [0,T]$, $Y_t \leq \widetilde Y_t$.

\begin{lem}\label{lem:upper_bound}
With the assumptions of Theorem \ref{thm:exist_min_sol}, we prove:
$$\forall t \in [0,T], \ \widetilde{Y}_{t} \leq \left( \frac{1}{q(T-t)} \right)^{\frac{1}{q}}.$$
\end{lem}
\begin{proof}
For every $0 < h < T$, we define on $[0,T-h]$
$$\Lambda_{h}(t) = \left( \frac{1}{q(T-h-t)} \right)^{\frac{1}{q}}.$$
$\Lambda_{h}$ is the solution of the ordinary differential equation:
$$\Lambda_{h}'(t)=(\Lambda_{h}(t))^{1+q}$$ 
with final condition $\Lambda_{h}(T-h) = + \infty$. But on the interval $[0,T-h]$, $(\widetilde{Y}, \widetilde{Z})$ is a solution of the BDSDE \eqref{eq:sing_BDSDE} with final condition $\widetilde{Y}_{T-h}$. From the assumptions $\widetilde{Y}_{T-h}$ is in $L^{2}(\Omega)$, so is finite a.s. Now we take the difference between $\widetilde{Y}$ and $\Lambda_{h}$ for all $0 \leq t \leq s < T-h$:
\begin{eqnarray*}
 \hat Y_t =  \widetilde{Y}_{t} -\Lambda_{h}(t)  & = &  \widetilde{Y}_{s} -\Lambda_{h}(s) - \int_{t}^{s} \left[  \left( \widetilde{Y}_{r}  \right)^{1+q} - \Lambda_{h}(r)^{1+q} \right] dr \\
& + & \int_t^s \left[g(r,\widetilde{Y}_r,\widetilde{Z}_r) - g(r,\Lambda_{h}(r),0)\right] \olaB - \int_{t}^{s} \widetilde{Z}_{r} dW_{r} .
\end{eqnarray*}
Recall that $g(t,y,0)=0$. We apply It\^o's formula to $(\hat Y_t^+)^2$ between $t$ and $s$:
\begin{eqnarray*}
(\hat Y_t^+)^2 & \leq &(\hat Y_s^+)^2 - 2 \int_t^s \hat Y_r^+ \left( \Lambda_{h}(r)^{1+q} - \left( \widetilde{Y}_{r} \right)^{1+q} \right)  dr \\
& + &  2 \int_t^s \hat Y_r^+ g(r,\widetilde{Y}_r,\widetilde{Z}_r)\overleftarrow{dB_r} - 2 \int_t^s \hat Y_r^+ \widetilde Z_{r}dW_{r}\\
& + & \int_t^s \ind_{\hat Y_r > 0} |g(r,\widetilde{Y}_r,\widetilde{Z}_r)- g(r,\Lambda_{h}(r),0)|^2dr - \int_t^s\ind_{\hat Y_r > 0} |\widetilde Z_r|^2 dr 
\end{eqnarray*}
The generator $f$ of this BDSDE satisfies Condition \eqref{ineq:monoton} with $\mu=0$, and $g$ satisfies \eqref{eq:property_Lip_g}. Thus 
\begin{eqnarray*}
(\hat Y_t^+)^2 &\leq &(\hat Y_s^+)^2 + K_g \int_t^s(\hat Y_r^+)^2  dr - (1-\eps) \int_t^s\ind_{\hat Y_r > 0} |\widetilde Z_r|^2 dr  \\
& + &  2 \int_t^s \hat Y_r^+ g(r,\widetilde{Y}_r,\widetilde{Z}_r)\overleftarrow{dB_r} - 2 \int_t^s \hat Y_r^+ \widetilde Z_{r}dW_{r}.
\end{eqnarray*}
We take the expectation of both sides. Since $(\widetilde Y,\widetilde Z)$ is in $\cE^2(0,s)$, the martingale part disappears and we deduce that:
\begin{eqnarray*}
\E (\hat Y_t^+)^2 &\leq  &\E (\hat Y_s^+)^2 + K_g  \int_t^s \E (\hat Y_r^+)^2 dr .
\end{eqnarray*}
By Gronwall's inequality we obtain:
$$\E (\hat Y_t^+)^2 \leq e^{K_g(s-t)} \E (\hat Y_s^+)^2.$$
Remark that for any $0 \leq t \leq T-h$
$$0 \leq \hat Y_t^+ \leq \sup_{0\leq t\leq T-h} \widetilde Y_t = \Xi_{T-h}.$$
Since $\widetilde Y\in \bS^2(0,T;\R_+)$, $\Xi_{T-h} \in L^2(\Omega)$. By dominated convergence theorem  as $s$ goes to $T-h$:
$$\E (\hat Y_t^+)^2 \leq e^{K_g(T-h-t)} \E (\hat Y_{T-h}^+)^2 = 0.$$
Thus $\widetilde Y_t \leq \Lambda_{h}(t)$ for all $t \in [0,T-h]$ and for all $0 < h < T$. So it is clear that for every $t \in [0,T]$:
$$\widetilde{Y}_{t} \leq \left( \frac{1}{q(T-t)} \right)^{\frac{1}{q}}.$$
This achieves the proof of the Lemma.
\end{proof}

Let us prove now minimality of our solution. We will prove that $\widetilde{Y}$ is greater than $Y^{n}$ for all $n \in \N$, which implies that $Y$ is the minimal solution. Let $(Y^{n},Z^{n})$ be the solution of the BSDE \eqref{eq:sing_BDSDE} with $\xi \wedge n$ as terminal condition. By comparison with the solution of the same BSDE with the deterministic terminal data $n$:
$$Y^{n}_{t} \leq \left( \frac{1}{q(T-t)+1/n^{q}} \right)^{1/q} \leq n.$$
Between the instants $0 \leq t \leq s < T$:
\begin{eqnarray} \nonumber
\hat Y_t =Y^{n}_{t} - \widetilde{Y}_{t} & = & \left(Y^{n}_{s} - \widetilde{Y}_{s}  \right)  - \int_{t}^{s} \left( (\widetilde{Y}_{r})^{1+q}  - (Y^{n}_{r})^{1+q}  \right) dr - \int_{t}^{s} \left( Z^{n}_{r} - \widetilde{Z}_{r}  \right) dW_{r} \\ \nonumber
&+& \int_t^s \left[ g(r,Y^n_{r},Z^n_{r}) - g(r,\widetilde{Y}_{r},\widetilde{Z}_{r})  \right] \olaB.
\end{eqnarray}
Once again we apply It\^o's formula to $(\hat Y_t^+)^2$:
\begin{eqnarray*}
(\hat Y^+_{t})^2 & \leq & (\hat Y^+_{s})^2 - 2 \int_t^s \hat Y_r^+ \left(  \left( Y^n_{r} \right)^{1+q}-(\widetilde Y_r)^{1+q}  \right)  dr \\
& + &  2 \int_t^s \hat Y_r^+ \left[ g(r,Y^n_r,Z^n_r) - g(r,\widetilde{Y}_r,\widetilde{Z}_r) \right] \overleftarrow{dB_r} - 2 \int_t^s \hat Y_r^+ \hat Z_{r}dW_{r}\\
& + & \int_t^s \ind_{\hat Y_r > 0} |g(r,Y^n_r,Z^n_r)-g(r,\widetilde{Y}_r,\widetilde{Z}_r)|^2dr - \int_t^s\ind_{\hat Y_r > 0} |\hat Z_r|^2 dr 
\end{eqnarray*}
and we deduce that 
$$\E(\hat Y^+_{t})^2 \leq e^{K_g(s-t)} \E(\hat Y^+_{s})^2.$$
Since $0\leq \hat Y^+_t \leq Y^n_t \leq n$, by dominated convergence theorem, we can take the limit as $s$ goes to $T$ and we obtain that $\E(\hat Y^+_{t})^2 = 0$. 

\section{Limit at time $T$ by localization technic} \label{sect:cont}

From now, the process $Y$ is continuous on $[0,T[$ and we define $Y_{T} = \xi$. The main difficulty will be to prove the continuity at time $T$. We have already proved that
 $$\liminf_{t \to T} Y_{t} = \lim_{t \to T} Y_t$$ 
and remarked that $Y$ is lower semi-continuous on $[0,T]$:
\begin{equation*} 
\xi \leq \liminf_{t \to T} Y_{t}.
\end{equation*}
In this paragraph we prove that the inequality in \eqref{liminf} is in fact an equality, i.e.
\begin{equation*}
\xi = \liminf_{t \to T} Y_{t}.
\end{equation*}
Note that the only remaining problem is on the set $\cR=\{\xi < +\infty\}$.

From now on, the conditions of Theorem \ref{thm:continuity_T} hold. In particular the terminal condition $\xi$ is equal to $h(X_T)$, where $h$ is a function defined on $\R^d$ with values in $\overline{\R^{+}}$. We denote by $\cS = \left\{ h=+\infty \right\}$ the closed set of singularity and $\cR$ its complement. In order to prove continuity at time $T$, we will show that for any function $\tet$ of class $C^{2}(\R^{d};\R^{+})$ with a compact support strictly included in $\cR = \left\{ h <+ \infty \right\}$ and for any $t \in [0,T]$, we have 
\begin{eqnarray} \label{eq:localization} 
\E(\xi \tet(X_{T})) & = & \E(Y_{t} \tet(X_{t})) + \E \int_{t}^{T} \tet(X_{r}) (Y_{r})^{1+q} dr + \E \int_{t}^{T} Y_{r}
\ope \theta (X_{r}) dr \\ \nonumber
& + & \E \int_{t}^{T} Z_{r} . \nabla \tet (X_{r}) \sigma(r,X_{r}) dr
\end{eqnarray}
with suitable integrability conditions on the last three terms in the right-hand side. Here $\ope$ is the operator:
\begin{equation} \label{eq:operator}
\ope = \frac{1}{2} \sum_{i,j} (\sigma \sigma^{*})_{ij} (t,x) \frac{\partial^{2}}{\partial x_{i} \partial x_{j}} + 
\sum_{i} b_{i}(t,x) \frac{\partial}{\partial x_{i}} = \frac{1}{2} \tr \left( \sigma \sigma^{*} (t,x) D^{2} \right) + b(t,x).\nabla;
\end{equation}
where in the rest of the paper, $\nabla$ and $D^{2}$ will always denote respectively the gradient and the Hessian matrix w.r.t. the space variable. If we let $t$ go to $T$ in Equality \eqref{eq:localization} and if we apply Fatou's lemma, we have:
\begin{equation} \label{eq:conclineq}
\E \left[ \xi \tet(X_{T}) \right]  =  \lim_{t \to T} \E \left[  Y_{t} \tet(X_{t}) \right] \geq  \E \left[ \left( \liminf_{t \to T} Y_{t} \right) \tet(X_{T}) \right]. 
\end{equation}
Note that we need the suitable estimates on the last three terms in \eqref{eq:localization}. Now recall that we already know \eqref{liminf}. Hence, the inequality in \eqref{eq:conclineq} is in fact a equality, i.e.
$$\E \left[ h(X_{T}) \tet (X_{T}) \right] = \E \left[ \tet (X_{T}) \left( \liminf_{t \to T} Y_{t} \right) \right].$$ 
And with \eqref{liminf} once again, we conclude that:
$$\lim_{t \to T} Y_t = \liminf_{t \to T} Y_{t} = h(X_{T}), \quad \prb- \mbox{a.s. on } \left\{ h(X_{T}) < \infty \right\}.$$
 
In the next subsections we prove that \eqref{eq:localization} holds. As in \cite{popi:06} the proof depends on the value of $q$ and we distinguish $q>2$ where no other assumption is needed (the non linearity is ``strong enough'') and $q \leq 2$ where we have to add additional conditions. Moreover the arguments are almost the same as in \cite{popi:06}, thus technical details will be skip here. 

Let $\phii$ be a function in the class $C^{2} \left(\R^{d}\right)$ with a compact support. Let $(Y,Z)$ be the solution of the BDSDE \eqref{eq:sing_BDSDE} with the final condition $\zeta \in L^{2}(\Omega)$. For any $t \in [0,T]$: 
\begin{eqnarray*} 
Y_{t} \phii(X_{t}) & = & Y_{0} \phii(X_{0}) + \int_{0}^{t} \phii(X_{r}) \left[Y_{r}|Y_{r}|^{q}  dr - g(r,Y_r,Z_r) \overleftarrow{dB_r}+ Z_{r}.dW_{r} \right] \\
& & + \int_{0}^{t} Y_{r} d(\phii(X_{r}))  + \int_{0}^{t} Z_{r} . \nabla \phii (X_{r}) \sigma(r,X_{r}) dr \\
& = & Y_{0} \phii(X_{0}) + \int_{0}^{t} \phii(X_{r}) Y_{r} |Y_{r}|^{q}dr + \int_{0}^{t} Z_{r} . \nabla \phii (X_{r}) \sigma(r,X_{r}) dr\\
& & +\int_{0}^{t} Y_{r} \ope \phii (X_{r}) dr + \int_{0}^{t} \left( Y_{r} \nabla \phii (X_{r}) \sigma(r,X_{r}) + \phii (X_{r}) Z_{r} \right). dW_{r} \\
&& -\int_{0}^{t} \phii(X_{r}) g(r,Y_r,Z_r) \overleftarrow{dB_r}
\end{eqnarray*}
where $\ope$ is the operator defined by \eqref{eq:operator}. Taking the expectation:
\begin{eqnarray} \label{eq:derivation}
\E(Y_{t} \phii(X_{t})) &=& \E (Y_{0} \phii(X_{0})) +  \E\int_{0}^{t} \phii(X_{r}) Y_{r} |Y_{r}|^{q} dr \\ \nonumber
& & + \E \int_{0}^{t} Z_{r} . \nabla \phii (X_{r}) \sigma(r,X_{r}) dr + \E \int_{0}^{t} Y_{r} \ope \phii (X_{r}) dr.
\end{eqnarray}
Note that the generator $g$ does not appear in this expression. Hence no extra assumption will be added in the case $q > 2$.

Let $U$ be a bounded open set with a regular boundary and such that the compact set $\overline{U}$ is included in $\cR$. We denote by $\Phi = \Phi_{U}$ a function which is supposed to belong to $C^{2}(\R^{d};\R_{+})$ and such that $\Phi$ is equal to zero on $\R^{d} \setminus U$, is positive on $U$. Let $\al$ be a real number such that 
$$\al > 2(1+1/q).$$
For $n \in \N$, let $(Y^{n},Z^{n})$ be the solution of the BSDE \eqref{eq:sing_BDSDE} with the final condition $(h \wedge n)(X_{T})$. The equality \eqref{eq:derivation} with $t=T$ becomes:
\begin{eqnarray} \nonumber
\E(Y^{n}_{T} \Phi^{\al}(X_{T})) & = & \E(Y^{n}_{0} \Phi^{\al}(X_{0})) +  
 \E \int_{0}^{T} Z^{n}_{r}. \nabla (\Phi^{\al}) (X_{r}) \sigma(r,X_{r}) dr \\  \label{eq:deriv}
& + & \E \int_{0}^{T} \Phi^{\al}(X_{r}) (Y^{n}_{r})^{1+q}  dr + \E \int_{0}^{T}  Y^{n}_{r} \ope (\Phi^{\al})(X_{r}) dr.
\end{eqnarray}
\begin{lem}\label{lem:tech_estim_1}
Let $p$ be such that 
$$\frac{1}{p} + \frac{1}{1+q} = 1.$$
Then 
$$\Phi^{-\al (p-1)} \left| \ope (\Phi^{\al}) \right|^{p} \in L^{\infty}([0,T] \times \R^{d}).$$
\end{lem}
Lemma \ref{lem:tech_estim_1} and H\"older inequality show that there exists a constant $C$ such that 
\begin{equation} \label{eq:growth_comp_explos_term}
\forall n \in \N, \quad \E \int_{0}^{T} \left| Y^{n}_{r} \ope (\Phi^{\al})(X_{r}) \right| dr \leq C  \left[ \E \int_{0}^{T} \Phi^{\al}(X_{r}) (Y^{n}_{r})^{1+q}  dr \right]^{1/(1+q}.
\end{equation}
We distinguish the case $q>2$ and $q\leq 2$ in order to control the term containing $Z$ in \eqref{eq:deriv}.

\subsection{Proof of \eqref{eq:localization} if $q>2$}

Proposition \ref{prop:sharp_estim_Z} and the Cauchy-Schwarz inequality prove immediatly the following result. 
\begin{lem}[Case $q > 2$]\label{lem:tech_estim_2}
If $q > 2$, then there exists a constant $C=C(q,\Phi,\al,\sigma)$ such that for all $n \in \N$:
\begin{equation} \label{eq:boundqgreater2}
\E \int_{0}^{T} \left| Z^{n}_{r}. \nabla (\Phi^{\al}) (X_{r}) \sigma(r,X_{r})  \right|  dr \leq C.
\end{equation}
\end{lem}
Lemmata \ref{lem:tech_estim_1} and \ref{lem:tech_estim_2} and Equality \eqref{eq:deriv} imply the next result.
\begin{lem} \label{lem:estimL1}
The sequence $\Phi^{\al}(X) (Y^{n})^{1+q}$ is a bounded sequence in $L^{1}(\Omega \times [0,T])$ and with the Fatou lemma, $Y^{1+q}\Phi^{\al}(X)$ belongs to $L^{1}(\Omega \times [0,T])$:
\begin{equation*} 
\E \int_{0}^{T} Y_{r}^{1+q}\Phi^{\al}(X_{r}) dr < + \infty.
\end{equation*}
\end{lem}
\begin{proof}
In \eqref{eq:deriv}, since the support of $\Phi$ is in $F_\infty^c$ and from the previous lemma, the first three terms are bounded w.r.t. $n$ by some constant $C$. If the sequence is not bounded, from inequality \eqref{eq:growth_comp_explos_term}, there is a contradiction.
\end{proof}

Now we prove Equality \eqref{eq:localization}. Let $\tet$ be a function of class $C^{2}(\R^{d};\R^{+})$ with a compact support strictly included in $\cR = \left\{ h < + \infty \right\}$. There exists a open set $U$ s.t. the support of $\tet$ is included in $U$ and $\overline{U} \subset \cR$. Let $\Phi = \Phi_{U}$ be the previously used function. Let us recall that $\al$ is strictly greater than $2(1+1/q) > 2$. Thanks to a result in the proof of the lemma 2.2 of \cite{marc:vero:99}, there exists a constant $C=C(\tet,\alpha)$ such that:  
$$|\tet| \leq C \Phi^{\al} \ , \ |\nabla \tet| \leq C \Phi^{\al-1} \ \mbox{and}\ \left\| D^{2} \tet \right\| \leq C \Phi^{\al-2}.$$
Using Lemma \ref{lem:estimL1} and the monotone convergence theorem, we have
$$\lim_{n\to + \infty} \E \int_t^T (Y^n_{r})^{1+q} \tet(X_r) dr = \E \int_t^T (Y_{r})^{1+q} \tet(X_r) dr,$$
with 
\begin{equation} \label{eq:estim_L1_1}
\E \int_{0}^{T} Y_{r}^{1+q}\tet(X_{r}) dr \leq C.
\end{equation}
We can do the same calculations using the previously given estimations on $\tet$, $\nabla \tet$ and $D^{2} \tet$ in terms of power of
$\Phi^{\al}$ and H\"{o}lder's inequality: 
\begin{equation*} 
\Phi^{-\al (p-1)} \left| \ope \theta \right|^{p} \in L^{\infty}([0,T] \times \R^{d}).
\end{equation*}
Now we can write:
$$Y^{n}_{r} \ope \theta(X_{r}) = \left( Y^{n}_{r} \Phi^{\al/(1+q)} \right) \left( \Phi^{-\al/(1+q)} \ope \theta(X_{r}) \right) = \left(
Y^{n}_{r} \Phi^{\al/(1+q)} \right) \left( \Phi^{-\al(p-1)/p} \ope \theta(X_{r}) \right).$$
The sequence $Y^{n}\Phi^{\al/(1+q)}=Y^{n}\Phi^{\al (1-1/p)}$ is a bounded sequence in $L^{1+q}(\Omega \times [0,T])$ (see
Lemma \ref{lem:estimL1}). Therefore using a weak convergence result and extracting a subsequence if necessary, we can pass to the limit in the term:
$$\E \int_{t}^{T} Y^{n}_{r} \ope \theta(X_{r}) dr,$$
with 
\begin{equation} \label{eq:estim_L1_2}
\E \int_{0}^{T} \left| Y_{r} \ope \theta(X_{r}) \right| dr \leq C.
\end{equation}

Recall the estimation \eqref{ineq:sharp_estim_Z}: there exists a constant $C$ such that for all $n \in \N$:
\begin{equation*}
\E \int_{0}^{T} | Z^{n}_{r} |^{2} (T-r)^{\frac{2}{q}}dr \leq C.
\end{equation*}
Hence, there exists a subsequence, which we still denote $Z^{n} (T-r)^{1/q}$, and which converges weakly in the space $L^{2}\left( \Omega \times (0,T), d \prb \times dt; \R^{d} \right)$ to a limit, and the limit is $Z (T-r)^{1/q}$, because we already know that $Z^{n}$ converges to $Z$ in $L^{2} \left( \Omega \times (0,T-\delta) \right)$ for all $\delta > 0$. $\nabla \tet (X) \sigma(.,X) (T- \ .)^{-1/q}$ is $L^{2} \left( \Omega \times (0,T) \right)$, because $\tet$ is compactly supported and $q > 2$. Therefore,
\begin{equation*} 
\lim_{n \to + \infty} \E \int_{t}^{T} Z^{n}_{r}. \nabla \tet (X_{r}) \sigma(r,X_{r}) dr = \E \int_{t}^{T} Z_{r}. \nabla \tet (X_{r}) \sigma(r,X_{r}) dr. 
\end{equation*}
And Lemma \ref{lem:tech_estim_2} shows that
\begin{equation} \label{eq:estim_L1_3}
\E \int_{0}^{T} \left| Z_{r}. \nabla \tet (X_{r}) \sigma(r,X_{r})  \right|  dr \leq C.
\end{equation}

To conclude we write Equality \eqref{eq:derivation} for $(Y^n,Z^n)$ and $\tet$ and we pass to the limit. This gives Equality \eqref{eq:localization} with the three estimates \eqref{eq:estim_L1_1}, \eqref{eq:estim_L1_2} and \eqref{eq:estim_L1_3}. 
 
\subsection{Proof of \eqref{eq:localization} if $q\leq2$}
 
If we just assume $q >0$, Lemma \ref{lem:tech_estim_2} does not hold anymore. In other words our previous control on the term containing $Z$ in \eqref{eq:localization} fails. But if we are able to prove that there exists a function $\psi$ such that for $0 < t \leq T$:
\begin{equation}\label{eq:malliavin_int_part}
\E \int_{t}^{T} Z^{n}_{r}. \nabla \tet (X_{r}) \sigma(r,X_{r}) dr = \E \int_{t}^{T} Y^{n}_{r} \psi(r,X_{r}) dr,
\end{equation}
then we apply again the H\"{o}lder inequality in order to control 
$$\E \int_{t}^{T} Y^{n}_{r} \psi(r,X_{r}) dr \quad \mbox{by} \quad \E \int_{t}^{T} (Y^{n}_{r})^{1+q} \Phi^{\al}(X_{r}) dr.$$
Note that we will add the following condition on $g$. From now on $g$ is a measurable function defined on $[0,T]\times \R^d \times \R \times \R^k$ such that the Lipschitz property \eqref{eq:property_Lip_g} holds and for any $(t,x,x',y,z) \in [0,T]\times (\R^d)^2 \times \R \times \R^m$
\begin{equation}\label{eq:lip_cond_g_x} \tag{A6}
 |g(t,x,y,z) - g(t,x',y,z)| \leq K_g |x-x'|.
\end{equation}

First we will prove the next result. In fact it is a quite straightforward modification of Proposition 2.3 in \cite{pard:peng:94} or Proposition 4.2 in \cite{bach:lasm:mato:13}. Let us denote by $\mathcal{B}^2(0,T;\mathbb{D}^{1,2})$ the set of processes $(Y,Z)$ such that $Y \in \mathbb{S}^2([0,T])$, $Z \in \mathbb{H}^2(0,T)$ and $Y_t$ and $Z_t$ belong to $\mathbb{D}^{1,2}$ with
$$\E \left[  \int_0^T ( |Y_t|^2 + |Z_t|^2) dt + \int_0^T \int_0^T (|D_sY_t|^2 + |D_s Z_t |^2) dtds \right] < + \infty.$$
\begin{lem}
Assume that $(Y,Z)$ is solution of the BSDE \eqref{eq:sing_BDSDE} with terminal condition $\xi = h(X_T)$:
\begin{equation*}
Y_{t}  =  h(X_T) - \int_{t}^{T} Y_{s}|Y_s|^q ds + \int_t^T g(s,X_s,Y_s,Z_s) \overleftarrow{dB_s}  - \int_{t}^{T} Z_{s} dW_{s}
\end{equation*}
where $h$ is a bounded Lipschitz function on $\R^d$, $g$ satisfies the previous Lipschitz condition and $X_s \in \mathbb{D}^{1,2}$ for every $s \in [0,T]$. 

\begin{enumerate}
\item Then $(Y,Z) \in \mathcal{B}^2(0,T;\mathbb{D}^{1,2})$ and for all $1 \leq i \leq d$, $\left\{ D_{s}^{i}Y_{s}, 0 \leq s \leq T \right\}$ is a version of $\left\{ (Z_{s})_{i}, 0 \leq s \leq T \right\}$. $(Z_{s})_{i}$ denotes the i-th component of $Z_{s}$. Here, $D_{s}^{i}Y_{s}$ has the following sense:  
$$D_{s}^{i}Y_{s} = \lim_{\underset{r < s}{r \rightarrow s}} D_{r}^{i}Y_{s}.$$
\item There exist two random fields $u$ and $v$ such that:
$$ Y_t=u(t,X_t)\quad and\quad Z_t=v(t,X_t)$$
and for any $(t,x)\in [0,T]\times\mathbb{R}^m$, $u(t,x)$ and $v(t,x)$ are $\mathcal{F}^B_{t,T}$-measurable.
\end{enumerate}
\end{lem}
\begin{proof}
We just sketch the proof. The technical arguments can be found in \cite{bach:lasm:mato:13} or \cite{elka:peng:quen:97}. It is known (see e.g. proof of Theorem 1.1 in \cite{pard:peng:94}) that the solution $(Y,Z)$ of the previous BSDE is obtained by passing to the limit in the following iteration scheme:
$$ (Y^0,Z^0)=(0,0)$$
$$ Y^{m+1}_t=h(X_T)-\int_t^T Y_s^m|Y_s^m|^{q}ds +\int_t^T g(s,X_s,Y^m_s,Z^m_s)\overleftarrow{dB_s} -\int_t^T Z^{m+1}_s dW_s.$$
Our goal here is to show by induction that  for all $m \in \N$ $(Y^m,Z^m)  \in  \mathcal{B}^2(0,T;\mathbb{D}^{1,2})$ and that $(Y^m,Z^m)$ converges in $L^2([0,T] , \mathbb{D}^{1,2} \times \mathbb{D}^{1,2})$ to $(Y,Z)$. 

It is clear that this is true at step 0 since constants are Malliavin differentiable. Let us suppose now that $(Y^m,Z^m)  \in  \mathcal{B}^2(0,T;\mathbb{D}^{1,2})$. By Proposition 1.2.4 in \cite{nual:95} we know that $h(X_T)\in\mathbb{D}^{1,2}$ and $g(s,X_s,Y^m_s,Z^m_s)\in\mathbb{D}^{1,2}$, since $h$ and $g$ are supposed Lipschitz continuous and $X_T$, $Y^m_s$ and $Z^m_s$ are in $\mathbb{D}^{1,2}$. With Lemma 4.2 in \cite{bach:lasm:mato:13} we obtain that $\int_t^T g(s,X_s,Y^m_s,Z^m_s)d\overleftarrow{B}_s\in\mathbb{D}^{1,2}$. Moreover, since $h$ is bounded, $Y^m$ is also bounded (see \eqref{eq:bound_Y}) and $x\mapsto x|x|^{q}\in C^1(\R)$. Therefore $Y_t^m |Y_t^m|^{q}\in\mathbb{D}^{1,2}$. Then following the arguments of section 7.1 in \cite{bach:lasm:mato:13}, we obtain that $(Y^{m+1},Z^{m+1})  \in  \mathcal{B}^2(0,T;\mathbb{D}^{1,2})$ with $D_s Y^{m+1}_t = D_s Z^{m+1}_t = 0$ for $0 \leq t \leq s \leq T$ and for $0 \leq s \leq t \leq T$
\begin{eqnarray*}
D_s Y^{m+1}_t & = & H^x_T (D_s X_T) -(q+1) \int_t^T |Y^m_u|^q D_s Y^m_u du \\
&+ &  \int_t^T  \left[ G^{x,m}_u D_s X_u +G^{y,m}_u D_s Y^m_u +G^{z,m}_u D_s Z^m_u \right] d\overleftarrow{B}_u - \int_t^T D_sZ^{m+1}_u dW_u
\end{eqnarray*}
where $H^x_T$, resp. $G^{x,m}_u$, $G^{y,m}_u$ and $G^{z,m}_u$ are four bounded random variables  with $H^x$ (resp. $G^x_u$, $G^y_u$ and $G^z_u$) is $\tri^W_T$ (resp. $\tri_u$) measurable. The bound depends only on the Lipschitz constant of $h$ and $g$. Using the same arguments as in \cite{bach:lasm:mato:13} we can prove that $G^{x,m}$, $G^{y,m}$ and $G^{z,m}$ converge to bounded processes and $(Y^m,Z^m)$ converges in $L^2([0,T] , \mathbb{D}^{1,2} \times \mathbb{D}^{1,2})$ to $(Y,Z)$. Now for $0 < s < t$ 
\begin{eqnarray*}
D_s Y_t & = & Z_s + (q+1) \int_s^t |Y_u|^q D_s Y_u du \\
&- &  \int_s^t  \left[ G^{x}_u D_s X_u +G^{y}_u D_s Y_u +G^{z}_u D_s Z_u \right] d\overleftarrow{B}_u + \int_s^t D_sZ_u dW_u. 
\end{eqnarray*}
We then pass to the limit as $s$ goes to $t$ to obtain the desired result.

For the second part, we will show that there exists two random fields $u^m$ and $v^m$, such that for any $(t,x)\in [0,T]\times\mathbb{R}^m$, $u^m(t,x)$ and $v^m(t,x)$ are $\mathcal{F}^B_{t,T}$-measurable and 
$$ Y^m_t=u^m(t,X_t)\quad \mbox{and} \quad Z^m_t=v^m(t,X_t).$$
It is clear that this is true at step 0 by taking $u^0=v^0=0$. Let us suppose now that there exists two functions satisfying the measurability as stated in the proposition such that at step $m$ our processes verify:
$$Y^m_t=u^m(t,X_t)\quad \mbox{and} \quad Z^m_t=v^m(t,X_t).$$
We can write then:
\begin{eqnarray*}
Y^{m+1}_t & = & h(X_T)-\int_t^T u^m(s,X_s) |u^m(s,X_s)|^{q}ds \\
& +& \int_t^T g(s,X_s,u^m(s,X_s),v^m(s,X_s))d\overleftarrow{B}_s -  \int_t^T Z^{m+1}_s dW_s
\end{eqnarray*}
We take the expectation with respect to $\mathcal{G}_t$, which leads to:
\begin{eqnarray*}
Y^{m+1}_t&=&\E^{\mathcal{G}_t}\left[ h(X_T)-\int_t^T u^m(s,X_s) |u^m(s,X_s)|^{q}ds \right. \\
&& \qquad \qquad \left.+\int_t^T g(s,X_s,u^m(s,X_s),v^m(s,X_s))d\overleftarrow{B}_s\right]. 
\end{eqnarray*}
By Markov property, there exists a function $u^{m+1}$ with $u^{m+1}(t,x)$ is $\mathcal{F}^B_{T}$-measurable and 
$$Y^{m+1}_t=u^{m+1}(t,X_t).$$
We know that $Y^{m+1}_t$ is $\mathcal{F}_t$-measurable and we deduce that $u^{m+1}(t,x)$ is in fact $\mathcal{F}^B_{t,T}$-measurable. The convergence of $u^m$ and $v^m$ can be obtained by the same arguments as in \cite{elka:peng:quen:97}, section 4. 
\end{proof}

The last assumption implies that $x\mapsto h(x) \wedge n$ is a bounded Lipschitz function on $\R^d$. Moreover these conditions imply that $X_s \in \mathbb{D}^{1,2}$. Therefore we can apply the previous proposition to $(Y^n,Z^n)$ to establish:

\begin{prop}
There exists a function $\psi$ such that for every $s\in [0,T]$
$$ \E \left[Z^n_s\nabla\theta(X_s)\sigma(s,X_s) \right]=\E \left[Y^n_s\psi(s,X_s)\right],$$
where $\psi$ is given by:
\begin{eqnarray} \nonumber
\psi(t,x) & = & -\sum_{i=1}^d (\nabla\theta(x)\sigma(t,x))_i\frac{\diver(p\sigma^i)(t,x)}{p(t,x)} \\ \label{eq:def_auxili_fct_psi}
&-& \tr(D^2\theta(x)\sigma\sigma^*(t,x)) -\sum_{i=1}^d\nabla\theta(x).\nabla\sigma^i(t,x)\sigma^i(t,x).
\end{eqnarray}
\end{prop}
\begin{proof}
We would like to apply here the integration by parts formula (Lemma 1.2.2 in \cite{nual:95}) which states:
$$\E[G \langle DF,h \rangle_H]=\E[-F\langle DG,h \rangle_H+FGW(h)],$$
where $H:=\mathbb{L}^2([0,T],\R^d)$ and $F$ and $G$ are two random variables in $\mathbb{D}^{1,2}$. But in order to use this we need to rewrite our expectation to make
appear the scalar product in $H$. Actually we have:
$$\E\left[ Z^n_t.\nabla\theta(X_t)\sigma(t,X_t)\right]=\E\left[D_t Y^n_t. \nabla\theta(X_t)\sigma(t,X_t)\right] = \sum_{i=1}^d \E\left[D^i_t Y^n_t (\nabla\theta \sigma)_i(X_t)\right] $$
where $(\nabla\theta \sigma)(X_t) = \nabla\theta(X_t)\sigma(t,X_t)$ and $(\nabla\theta \sigma)_i(X_t)$ denotes the $i$-th component of $(\nabla\theta \sigma)(X_t)$. As in the proof of Proposition 17 in \cite{popi:06} we can use the following approximation:
$$\mathbb{P}-a.s.\quad D^i_t Y^n_t=\lim_{j\rightarrow \infty} \langle DY^n_t,\nu^i_j \rangle_H = \lim_{j\rightarrow \infty}\int_0^T  j\bold{1}_{[t-\frac{1}{j},t[}(s) (D_s Y^n_t).e_i ds$$
$(e_i)_{i=1,\ldots,d}$ being the canonical basis of $\R^d$. 

Let $t\in[0,T]$, by integrations by parts formula we get:
\begin{eqnarray*}
\E[\langle DY^n_t,\nu^i_j \rangle_H (\nabla\theta \sigma)_i(X_t)]&=&\E\left[ Y^n_t (\nabla\theta \sigma)_i(X_t)\int_0^T \nu^i_j(s)dW_s \right]\\
&-&\E \left[ Y^n_t\int_0^T \nu^i_j(s)D_s((\nabla\theta \sigma)_i(X_t))ds \right].
\end{eqnarray*}
Let us consider the first term:
\begin{eqnarray*}
&& \E\left[ Y^n_t(\nabla\theta \sigma)_i(X_t) \int_0^T \nu^i_j(s)dW_s \right] = j \E \left[ Y^n_t (\nabla\theta \sigma)_i(X_t) \left(W^i_t-W^i_{t-\frac{1}{j}} \right) \right]\\
&& \qquad \qquad \qquad  = j \E \left[ u^n(t,X_t) (\nabla\theta \sigma)_i(X_t)  \left(W^i_t-W^i_{t-\frac{1}{j}} \right) \right]\\
&& \qquad \qquad \qquad  = j \E \left[ \E \left[ u^n(t,X_t) \bigg| \tri^W_t \right] (\nabla\theta \sigma)_i(X_t)  \left(W^i_t-W^i_{t-\frac{1}{j}} \right) \right]\\
&& \qquad \qquad \qquad  = j \E \left[ v^n(t,X_t) (\nabla\theta \sigma)_i(X_t)  \left(W^i_t-W^i_{t-\frac{1}{j}} \right) \right]
\end{eqnarray*}
since $u^n(t,x)$ is $\tri^B_{t,T}$-measurable and $B$ and $W$ are independent. Now
\begin{eqnarray*}
&& \E \left[ v^n(t,X_t) (\nabla\theta \sigma)_i(X_t)  \left(W^i_t-W^i_{t-\frac{1}{j}} \right) \right] \\
&& \quad = -\E \left[ v^n(t,X_t) (\nabla\theta \sigma)_i(X_t) \int_{t-\frac{1}{j}}^t\frac{\diver(\sigma^i p)(u,X_u)}{p(u,X_u)}du \right] 
\end{eqnarray*}
where $p$ is the density of $X$ and $\sigma^i$ is the $i$-th column of the matrix $\sigma$. As in \cite{popi:06}, we use Lemmas 3.1 and 4.1 in \cite{pard:86} with the same arguments. For convenience let us just recall that from Theorems 7 and 10 in \cite{aron:68}, Theorem II.3.8 of \cite{stro:88} and Theorem III.12.1 in \cite{lady:solo:ural:68}, the density $p(x;.,.)$ exists and satisfies:
\begin{itemize}
\item $p(x;.,.) \in L^2 (\delta,T; H^2)$ for all $\delta  > 0$;
\item $p$ is H\"older continuous in $x$ and satisfies the following inequality for $s \in ]0, T ]$: 
\begin{equation}\label{eq:Aronson_estim}
\frac{\exp \left( -C \frac{|y-x|^2}{s} \right)}{Cs^{m/2}} \leq p(x;s,y) \leq \frac{C \exp \left( -\frac{|y-x|^2}{Cs} \right)}{s^{m/2}};
\end{equation}
\item $y \mapsto \partial p/ \partial y_i(x;.,.)$ is H\"older continuous in $y$. 
\end{itemize}
In this proof we omit the variable x in $p(x ; ., .)$. The previous properties of $p$ ensure that $\diver(p\sigma^i)/p$ is well-defined and regular. Let $j$ go to $+\infty$, we have:
\begin{eqnarray*}
\E\left[D^i_t Y^n_t (\nabla\theta \sigma)_i(X_t)\right]  & =&  -\E \left[ Y^n_t (\nabla\theta \sigma)_i(X_t) \frac{\diver(\sigma^i p)(t,X_t)}{p(t,X_t)} \right]  \\
&-&\E\left[Y^n_t D_t(\nabla\theta(X_t)\sigma(t,X_t))\right]. 
\end{eqnarray*}
From the regularity assumptions on $\theta$ and $\sigma$ we can compute directly the last term to obtain:
$$\E\left[ Z^n_t.\nabla\theta(X_t)\sigma(t,X_t)\right]=\E\left[ Y^n_t\psi(t,X_t)\right]$$
with $\psi$ given by \eqref{eq:def_auxili_fct_psi}.
\end{proof}

As in the case $q > 2$, we have to prove that Equality \eqref{eq:localization} holds with suitable integrability conditions. Now Equation \eqref{eq:deriv} becomes:
\begin{eqnarray} \nonumber
&& \E(Y^{n}_{T} \Phi^{\al}(X_{T})) = \E(Y^{n}_{0} \Phi^{\al}(X_{0})) +  
 \E \int_{0}^{T} Z^{n}_{r}. \nabla (\Phi^{\al}) (X_{r}) \sigma(r,X_{r}) dr \\ \nonumber
&& \qquad + \E \int_{0}^{T} \Phi^{\al}(X_{r}) (Y^{n}_{r})^{1+q}  dr + \E \int_{0}^{T}  Y^{n}_{r} \ope (\Phi^{\al})(X_{r}) dr \\
&& =  \E(Y^{n}_{0} \Phi^{\al}(X_{0})) + \E \int_{0}^{T} \Phi^{\al}(X_{r}) (Y^{n}_{r})^{1+q}  dr + \E \int_{0}^{T}  Y^{n}_{r} \Psi_{\al}(r,X_{r}) dr 
\end{eqnarray}
with $\Psi_{\alpha}$ the following function: for $t \in ]0,T]$ and $x \in \R^{d}$ 
\begin{eqnarray*}
\Psi_{\al}(t,x) & = & \nabla (\Phi^{\al}) (x).b(t,x) - \frac{1}{2} \tr (D^{2}(\Phi^{\al})(x) \sigma \sigma^{*}(t,x)) \\
& - & \sum_{i=1}^{d} \left( (\nabla (\Phi^{\al}) (x) \sigma(t,x))_{i} 
\ \frac{\diver(p(t,x) \sigma^{i}(t,x))}{p(t,x)} \right) \\ \label{eq:derivalpha}
& - & \sum_{i=1}^{d} \left( \nabla (\Phi^{\al})(x).[\nabla\sigma^{i}(t,x) \sigma^{i}(t,x)] \right). 
\end{eqnarray*}
In \cite{popi:06} it is proved that for a fixed $\eps >0$ and $p=1+1/q$:
$$\Phi^{-\al(p-1)} \left| \Psi_{\al} \right|^{p} \in L^{\infty}([\eps,T] \times \R^{d}).$$
If it is true, then the last term in \eqref{eq:derivalpha} satisfies:
$$\E \int_{t}^{T} \left| Y^{n}_{r} \Psi_{\al}(r,X_{r}) \right| dr \leq C \left( \E \int_{t}^{T} \Phi^{\al}(X_{r}) (Y^{n}_{r})^{1+q} dr \right)^{\frac{1}{1+q}}$$ 
and the end of the proof will be the same as in the case $q > 2$.

\section{Link with SPDE's} \label{sect:SPDE}

In the introduction, we have said that there is a connection between doubly stochastic backward SDE whose terminal data is a function of the value at time $T$ of a solution of a SDE (or forward-backward system), and solutions of a large class of semilinear parabolic stochastic PDE. Let us precise this connection in our case.  

To begin with, we modify the equation \eqref{eq:eds}. We denote by $X^{t,x}$ the solution of the SDE \eqref{eq:eds2} with $b \in C^2_b$ and $\sigma \in C^3_b$. Therefore $b$ and $\sigma$ satisfy the assumptions (\ref{eq:lipcond})-(\ref{eq:growthcond}). We consider the following doubly stochastic BSDE for $t \leq s \leq T$:
\begin{equation} \label{eq:fbsde} 
Y^{t,x}_{s}  = h(X^{t,x}_{T}) - \int_{s}^{T} Y^{t,x}_{r}|Y^{t,x}_{r}|^{q} dr + \int_s^T g(r,X^{t,x}_r,Y^{t,x}_r,Z^{t,x}_r) \overleftarrow{dB_r}  - \int_{s}^{T} Z^{t,x}_{r} dW_{r},
\end{equation}
where $h$ is a function defined on $\R^{d}$ with values in $\R$. The two equations \eqref{eq:eds2} and \eqref{eq:fbsde} are called a forward-backward system. This system is connected with the stochastic PDE \eqref{eq:sing_SPDE}  with terminal condition $h$.

More precisely for any $n \in \N^*$, let $(Y^{n,t,x},Z^{n,t,x})$ be the solution of the BDSDE \eqref{eq:fbsde} with terminal condition $h(X^{t,x}_T)\wedge n$. We know that 
$$0 \leq Y^{n,t,x}_{s} \leq  \left( \frac{1}{q(T-s)+\frac{1}{n^{q}}} \right)^{\frac{1}{q}} \leq n.$$
And the generator $y \mapsto -y|y|^q$ is Lipschitz continuous on the interval $[-n,n]$. Moreover if we assume Assumption \eqref{eq:hyp3}, then $h \wedge n$ is a Lipschitz and bounded function on $\R^d$. Hence $h\wedge n$ belongs to $L^2(\R^d,\rho^{-1}(x) dx)$ provided the function $\rho^{-1} \in L^1(\R^d,dx)$. Since we have imposed that $g(t,x,y,0)=0$, now the all conditions in \cite{ball:mato:01} are satisfied. 
\begin{prop}[Theorem 3.1 in \cite{ball:mato:01}]
There exists a unique weak solution $u^n \in \cH(0,T)$ of the SPDE \eqref{eq:sing_SPDE}  with terminal function $h\wedge n$. Moreover $u^n(t,x) = Y^{n,t,x}_t$ and 
$$Y^{n,t,x}_s = u^n(s,X^{t,x}_s), \quad Z^{n,t,x}_s = (\sigma^* \nabla u^n)(s,X^{t,x}_s).$$
\end{prop} 
The space $\cH(0,T)$ is defined in Section \ref{sect:main_results}, Definition \ref{def:sobolev_space} and the notion of weak solution is precised in Section \ref{sect:mono_BDSDE}, Definition \ref{def:weak_sol_SPDE}. 

Remember that we have defined a process $(Y^{t,x},Z^{t,x})$ solution in the sense of the Definition \ref{definsolution} of the backward doubly stochastic differential equation \eqref{eq:fbsde} with singular terminal condition $h$ (see Theorem \ref{thm:exist_min_sol} and the beginning of the section \ref{sect:cont} on continuity at time $T$). The process $Y$ is obtained as the increasing limit of the processes $Y^n$: 
$$Y^{t,x}_s = \lim_{n \to +\infty} Y^{n,t,x}_s\quad \mbox{a.s.}.$$
Therefore we can define the following random field $u$ as follows:
$$u(t,x) = Y^{t,x}_t = \lim_{n \to +\infty} Y^{n,t,x}_t = \lim_{n \to +\infty} u^{n}(t,x).$$
Our aim is to prove Theorem \ref{thm:sol_SPDE}, that is, $u$ is also a weak solution of \eqref{eq:sing_SPDE} with the singular terminal condition $h$. For any $n$ we have a.s.
$$0\leq Y^{n,t,x}_s \leq \left( \frac{1}{q(T-s)+\frac{1}{n^{q}}} \right)^{\frac{1}{q}} \leq \left( \frac{1}{q(T-s)} \right)^{\frac{1}{q}}.$$
In particular for any $(t,x)$
$$0\leq u^{n}(t,x) \leq \left( \frac{1}{q(T-t)} \right)^{\frac{1}{q}}$$
and hence $u$ satisfies the same estimate. Thus $u$ is bounded on $[0,t]\times \R^d$ and in $L^2_\rho(\R^d)$. By dominated convergence theorem, for any $\delta > 0$, $u$ satisfies \eqref{eq:weak_sol_int_cond} and \eqref{eq:weak_sol_cont_cond} for any $0\leq s \leq t \leq T-\delta$. 

Moreover we have $Z^{n,t,x}_s = (\sigma^* \nabla u^n)(s,X^{t,x}_s)$ and from the proof on Theorem \ref{thm:exist_min_sol} we know that the sequence of processes $(Z^{n,t,x}_s, \ s \geq t)$ converges in $L^2((0,T-\delta)\times \Omega)$ for any $\delta > 0$ to $Z^{t,x}$. Hence the sequence $u^n$ converges in $\cH(0,t)$ to $u$. From Proposition \ref{prop:sharp_estim_Z} we have the a priori estimate \eqref{ineq:sharp_estim_Z}:
$$\E \int_{0}^{T} (T-s)^{2/q} | Z^{n,t,x}_{s} |^{2} ds \leq \frac{8+KT}{1-\eps} \left( \frac{1}{q} \right)^{2/q}$$
(as usual $Z^{n,t,x}_s = 0$ if $s < t$). Therefore we deduce
$$\E \int_{0}^{T} (T-s)^{2/q} | (\sigma^* \nabla u^n)(s,X^{t,x}_s) |^{2} ds \leq \frac{8+KT}{1-\eps} \left( \frac{1}{q} \right)^{2/q}.$$
We multiply each side by $\rho^{-1}(x)$, we integrate w.r.t. $x$ and we use Proposition 5.1 in \cite{ball:mato:01} to have:
$$\E \int_{\R^d} \int_{0}^{T} (T-s)^{2/q} | (\sigma^* \nabla u^n)(s,x) |^{2} \rho^{-1}(x) dx ds \leq C$$
where the constant $C$ does not depend on $n$. With the Fatou lemma we have the same inequality for $u$. 
$$\E \int_{\R^d} \int_{0}^{t} (T-s)^{2/q} | (\sigma^* \nabla u)(s,x) |^{2} \rho^{-1}(x) dx ds < \infty.$$

Now for every function $\Psi \in C^{1,\infty}_c([0,T] \times \R^d)$, $u^n$ satisfies \eqref{eq:weak_form_spde}, therefore for every $0\leq r \leq t < T$, $u^n$ satisfies also:
\begin{eqnarray}\label{eq:weak_form_sing_spde}
& &\int_r^t\int_{\R^d}u^n(s,x)\partial_s\Psi(s,x)dxds+\int_{\R^d}u^n(r,x)\Psi(r,x) dx -\int_{\R^d}u^n(t,x)\Psi(t,x)dx\\ \nonumber
&&\qquad - \frac{1}{2}\int_r^t \int_{\R^d}(\sigma^*\nabla u^n)(s,x)(\sigma^*\nabla\Psi)(s,x)dx ds\\ \nonumber
&&\qquad  -\int_r^t \int_{\R^d}u^n(s,x) \diver\left(\left(b-\widetilde{A}\right)\Psi\right)(s,x)dx ds\\ \nonumber
&&= -\int_r^t \int_{\R^d}\Psi(s,x)u^n(s,x)|u^n(s,x)|^{q} dx ds\\ \nonumber
&&\qquad +\int_r^t \int_{\R^d} \Psi(s,x) g(s,x,u^n(s,x),\sigma^*\nabla u^n(s,x)) dx \overleftarrow{dB_s} .
\end{eqnarray}
But using monotone convergence theorem or the convergence of $u^n$ to $u$ in $\cH(0,t)$, we can pass to the limit as $n$ goes to $+\infty$ in \eqref{eq:weak_form_sing_spde} and we obtain that $u$ is a weak solution of \eqref{eq:sing_SPDE}  on $[0,T-\delta]\times \R^d$ for any $\delta > 0$.

The only trouble concerns the behavior of $u$ near $T$. Since $u^n$ satisfies \eqref{eq:weak_form_spde} for any $n$, and by the integrability or regularity assumptions on $u^n$, $b$ and $\sigma$, we have
$$\lim_{t\to T} \int_{\R^d}u^n(t,x)\psi(x) dx = \int_{\R^d}(h(x)\wedge n)\psi(x)dx$$
for any function $\psi \in C^{\infty}_c(\R^d)$. Therefore by monotonicity
$$ \liminf_{t\uparrow T} \int_{\R^d}u(t,x)\psi(x)dx \geq  \liminf_{t\uparrow T} \int_{\R^d}u^n(t,x)\psi(x)dx = \int_{\R^d} (h(x)\wedge n) \psi(x)dx$$
for any $n$. Hence a.s.
\begin{equation}\label{eq:behavior_u_T}
 \liminf_{t\uparrow T} \int_{\R^d}u(t,x)\psi(x)dx \geq  \int_{\R^d} h(x) \psi(x)dx.
 \end{equation}
Our aim is to prove the converse inequality with the $\limsup$. In the second section we have proved Equation \eqref{eq:localization} with suitable integrability condition on all terms: 
\begin{eqnarray*} 
\E(h(X^{t,x}_T) \tet(X^{t,x}_{T})) & = & \E(u(t,x) \tet(x)) + \E \int_{t}^{T} \tet(X^{t,x}_{r}) (Y^{t,x}_{r})^{1+q} dr \\
& +&  \E \int_{t}^{T} Y^{t,x}_{r} \ope \theta (X^{t,x}_{r}) dr + \E \int_{t}^{T} Z^{t,x}_{r} . \nabla \tet (X^{t,x}_{r}) \sigma(r,X^{t,x}_{r}) dr
\end{eqnarray*}
for any smooth functions $\tet$ such that its compact support is strictly included in $\cR = \left\{ h <+ \infty \right\}$. If we integrate this w.r.t. $dx$ (no weight function $\rho$ is needed here since $\tet$ is of compact support) and we let $t$ go to $T$ we obtain:
\begin{eqnarray*} 
\lim_{t \to T} \int_{\R^d} \E(h(X^{t,x}_T) \tet(X^{t,x}_{T})) dx& = &\lim_{t \to T} \int_{\R^d} \E(u(t,x) \tet(x)) dx.
\end{eqnarray*}
By dominated convergence theorem this gives:
$$\lim_{t \to T} \E \left(  \int_{\R^d} u(t,x) \tet(x) dx \right) = \int_{\R^d} h(x) \tet(x) dx$$
for any function $\tet \in C^2_c(\R^d)$ with $\supp(\tet)\cap \cS = \emptyset$. Fatou's lemma implies that 
$$ \E \left(  \liminf_{t \to T} \int_{\R^d} u(t,x) \tet(x) dx \right) \leq \int_{\R^d} h(x) \tet(x) dx.$$

With Inequality \eqref{eq:behavior_u_T}, we obtain that for any $\psi \in C^{\infty}_c(\R^d)$ a.s.
$$\liminf_{t \to T} \int_{\R^d} u(t,x) \psi(x) dx = \int_{\R^d} h(x) \psi(x) dx.$$

\begin{rem}
If $g$ does not depend of $Z$ (or on $\nabla u$), and if $g \in C^{0,2,3}_b([0,T] \times \R^d \times R ; \R^d)$, then from \cite{buck:ma:01a}, $u^n$ is a stochastically bounded viscosity solution of the SPDE \eqref{eq:sing_SPDE} on $[0,T]\times \R^d$ and $u$ is also a stochastically bounded viscosity solution of the SPDE \eqref{eq:sing_SPDE} on $[0,T-\delta]\times \R^d$ for any $\delta > 0$.
\end{rem}

Now let $\widetilde u$ be a non negative weak solution of \eqref{eq:sing_SPDE} on $[0,T-\delta]\times \R^d$ for any $\delta > 0$. It means that for any $\delta > 0$, $\widetilde u\in \cH(0,T-\delta)$ and $\widetilde u$ satisfies \eqref{eq:weak_sol_int_cond} and \eqref{eq:weak_sol_cont_cond} and \eqref{eq:weak_form_sing_spde} on $[0,T-\delta]$. Moreover we assume that $\liminf_{t\to T} \widetilde u(t,x) \geq h(x)$ a.e. on $\Omega \times \R^d$. We follow the proof of uniqueness for Theorem 3.1 in \cite{ball:mato:01}. We define
$$\widetilde Y^{t,.}_s =\widetilde u(s,X^{t,.}_s), \quad \widetilde Z^{t,.}_s =\widetilde (\sigma^* \nabla u)(s,X^{t,.}_s).$$
Then by the same arguments as in \cite{ball:mato:01}, $(\widetilde Y^{t,x}_s ,\widetilde Z^{t,x}_s)$ solves the BDSDE \eqref{eq:sing_BDSDE} on any interval $[0,T-\delta]$. Moreover a.s.
$$\liminf_{s\to T} \widetilde Y^{t,x}_s \geq h(X^{t,x}_T) = \xi.$$ 
By Lemma \ref{lem:upper_bound} and the proof of minimality of the solution of the BDSDE (Theorem \ref{thm:exist_min_sol}), we deduce that for every $n$, a.s.
$$Y^{t,x,n}_s \leq \widetilde Y^{t,x}_s \leq \left( \frac{1}{q(T-s)}\right)^{1/q}.$$
Thus $u(t,x)\leq \widetilde u(t,x)$ a.e. on $\Omega \times [0,T] \times \R^d$. And $u$ is the minimal weak solution of the SPDE \eqref{eq:sing_SPDE} with singular terminal condition.

\bibliography{biblio_sing_SPDE}

\end{document}